\documentclass[12pt,a4paper]{amsart}
\usepackage{amscd}
\usepackage{mathtools}
\usepackage{color}
\usepackage{hyperref}
\usepackage[normalem]{ulem}
\usepackage{tikz}
\usepackage{tikz-cd}
\usepackage[enableskew]{youngtab}
\usepackage{graphicx,scalerel}

\newcommand{\R}{\mathbb R}
\newcommand{\Z}{\mathbb Z}
\newcommand{\C}{\mathbb C}

\newcommand{\cD}{\mathcal D}

\newcommand{\f}{\mathfrak f}

\newcommand{\g}{\mathfrak g}
\newcommand{\gl}{\mathfrak{gl}}

\newcommand{\fsl}{\mathfrak{sl}}
\newcommand{\fso}{\mathfrak{so}}
\newcommand{\hf}{\mathfrak h}
\newcommand{\lf}{\mathfrak l}

\newcommand{\of}{\mathfrak{so}}

\newcommand{\Hf}{\mathfrak{H}}

\newcommand{\om}{\omega}

\newcommand{\sll}{\mathfrak{sl}}

\newcommand{\Ad}{\operatorname{Ad}}
\newcommand{\ad}{\operatorname{ad}}

\newcommand{\Aut}{\operatorname{Aut}}
\newcommand{\aut}{\operatorname{aut}}

\newcommand{\Flag}{\operatorname{Flag}}

\newcommand{\im}{\operatorname{im}}
\newcommand{\gr}{\operatorname{gr}}
\newcommand{\Hom}{\operatorname{Hom}}

\newcommand{\Prol}{\operatorname{Prol}}

\newcommand{\diag}{\operatorname{diag}}

\DeclareRobustCommand{\amgiS}{\text{\reflectbox{$\Sigma$}}}

\newtheorem{thm}{Theorem}
\newtheorem{thmm}{Theorem}

\newtheorem{prop}{Proposition}

\theoremstyle{definition}
\newtheorem{df}{Definition}

\theoremstyle{remark}

\begin{document}

\title[Extrinsic Geometry and Differential Equations of $\fsl_3$-type]{Extrinsic Geometry and Linear Differential Equations of $\fsl_3$-type}
\author{Boris Doubrov}
\author{Tohru Morimoto}

\keywords{Extrinsic geometry, filtered manifold, osculating map, involutive systems of linear differential equations, extrinsic Cartan connection, Cayley's surface, adjoint variety}
\subjclass[2010]{Primary 53A55, 53C24; Secondary 53C30, 53D10. }

\begin{abstract}
As an application of the general theory on extrinsic geometry~\cite{DMM1},
we investigate extrinsic geometry in frag varieties and systems of linear PDE's for a class of special interest associated with the adjoint representation of $\fsl(3)$. We carry out a complete local classification of the homogeneous structures in this class.
As a result, we find 7 kinds of new systems of linear PDE's of second order on a 3-dimensional contact manifold each of which has a solution space of dimension 8.
Among them there are included a system of PDE's called contact Cayley's surface and one which has $\fsl(2)$ symmetry.
\end{abstract}

\maketitle

\maketitle
\tableofcontents

\section*{0. Introduction}

In our previous paper~\cite{DMM1} we have developed a general unified theory for extrinsic geometry in flag varieties and for geometry of linear differential equations.

In the present paper we apply it to a remarkable concrete class of extrinsic geometries and linear differential equations associated with the adjoint representation of $\mathfrak{sl}(3)$, and we carry through detailed studies on them. As one of the main goals, we then give a complete classification of the homogeneous structures in this class.

According to~\cite{DMM1}, it is the osculating maps 
\[
\varphi \colon (M, \f) \to L/L^0\subset \Flag(V, \phi)
\] 
that play a principal role in general extrinsic geometry,  where $(M, \f)$ is a filtered manifold, $\Flag(V,\phi )$ denotes the flag variety consisting of all the (descending) filtrations of a vector space $V$ isomorphic to a fixed filtration $\phi$, $L$ is a Lie subgroup of $GL(V)$, and $L^0=\phi^0L$ is the isotropy subgroup which fixes $\phi$.

Note that the frag variety $\Flag(V, \phi)$ is a homogeneous space $GL(V)/ \phi^0GL(V)$. More over it is a filtered manifold with an invariant tangential filtration defined by the induced natural filtration $\{\phi^p\mathfrak{gl}(V)\}$ of $\mathfrak{gl}(V)$.

We say that the map $\varphi$ is osculating if it satisfies: 
\[
\underline{\f^p}\,\underline{\varphi^q} \subset \underline{\varphi^{p+q}}
\]
which is equivalent to saying that $\varphi$ is a morphism of filtered manifolds.

We say also that to osculating maps $\varphi \colon(M, \f) \to L/L^0$ and $\varphi' \colon (M', \f') \to L/L^0$ are equivalent if there exist 
an isomorphism $h\colon (M,\f) \to (M',\f')$ of filtered manifolds and an element $a\in L$ such that $\Lambda_a \circ \varphi=\varphi'\circ h$. 

We remark that extrinsic geometry in frag varieties can be identified with the geometry of liner differential equations by virtue of a categorical isomorphism between the category of the osculating maps in frag varieties and that of weighted involutive systems  of linear differential equations~\cite{Mor2002,DMM1}.

The equivalence problem in these geometries is settled as follows: For an osculating map $\varphi$ there is associated, to each point $x \in M$, the first order approximation of $\varphi$, $\gr \varphi _x=\bigoplus \varphi^p_x / \varphi^{p+1}_x$, which is not only a graded vector space but also turns to be a $\gr\f_x$-module, and is called the symbol of $\varphi$ at $x$.

We say that an osculating map $\varphi\colon (M, \f) \to L/L^0\subset \Flag(V, \phi)$ has a constant symbol $(\g_-, V,\phi,L)$
if there exists a nilpotent graded Lie algebra $\g_-=\bigoplus_{p<0}\g_p$ represented in $\gr\mathfrak{gl}(V)$ as a graded subalgebra of $\lf_{-}$ such that $(\gr\f_x,\gr\varphi_x)$ is isomorphic  to $(\g_-,\gr V)$. 

We then consider the subcategory $\mathcal{EXG}(\sigma)$ of the osculating maps of constant symbol $\sigma =(\g_-,V,\phi, L)$.

For simplicity we may assume that the filtered Lie algebra $\lf$ corresponding to $L$ is graded. An important algebraic object which characterize the geometry is the relative prolongation $\g$ of $\g_-$ in $\lf$ defined to be the maximal graded subalgebra of $\lf$ having $\g_-$ as its negative part.
In~\cite{DMM1} we have given an algorithm to find a complete system of invariant $\chi$ for an osculating map $\varphi$ In $\mathcal{EXG}(\g_-,V,\phi;L)$ by constructing (semi-)canonically a series of bundles over $M$,
\[
	Q^{(0)}\leftarrow \cdots Q^{(i)} \leftarrow Q^{(i+1)}\cdots \leftarrow Q^{(k)}=Q
\]
with $Q^{(i)} \leftarrow Q^{(i+1)}$ being a principal fiber bundle with structure group $G_{i+1}$, $\operatorname{Lie}(G_{i+1})=\g_{i+1}$
and then by defining an $\lf$-valued 1-form $\omega$ on $Q$ which satisfies $d\omega + \tfrac12 [
\omega,\omega]=0$ and a vector valued function $\chi \colon Q \to \Hom(\g,\lf / \g)$.

We note that $Q$ is, in general, not a principal fiber bundle over the base space $M$, but ``a step wise principal'' bundle.
If $\g$ satisfied  the condition (C) (existence of auxiliary invariant complementary subspaces) then the bundle $Q$ becomes a principal fiber bundle over $M$ and the structure function $\chi$ takes its values only on $\Hom_{+}(\g_-,\lf / \g)$. 

Moreover, it holds in general without condition (C), that the structure function $\chi$ vanishes if and only if its part taking values in cohomology group $H^1_+(\g_-,\lf/\g)$ is identically 0.

If $\g \subset \lf$ is a simple graded Lie algebra and the representation of $\g$ on $\gr V$ is irreducible, then $\g$  satisfies the condition (C), and the related geometry is called extrinsic parabolic geometry.  

Thus the extrinsic geometry in flag varieties may be well understood by the unified principle described as above, on the other hand, each subcategory $\mathcal{EXG}(\g_-,V,\phi;L)$ has its on rich world according to the algebraic nature of its symbol. Similar study of extrinsic parabolic geometries, but in a different context, were initiated in~\cite{landrob,robles}.  

In the simplest case where $(\g,V)$ is $\mathfrak{sl}(2,\R)$ and its irreducible representation of dimension $k$, the corresponding extrinsic geometry is nothing but the geometry of curves in the projective space $P^{k-1}$ and the corresponding differential equations are linear ordinary differential equations for one unknown function of order $k$, which has been well-studied since 19 century~\cite{wilch}. Next example can be found also in the classical works of Wilzynski on the surfaces in $P^3$ of hyperbolic type~\cite{wilch2}, that is those having the non-degenerate second fundamental form of signature $(1,1)$.  This geometry of surface can be interpreted as the extrinsic geometry of osculating maps of the following symbol type:
Take $\g=\mathfrak{so}(2,2)=\mathfrak{sl}(2,\R)\times\mathfrak{sl}(2,\R)$ with the standard representation on $V=\R^4$ and with the standard grading of depth~1, $\g=\g_{-1}\oplus\g_0 \oplus \g_1$.  

In the present paper we consider the subcategory $\mathcal{EXG}(\sigma_{sl3})$, where the symbol $\sigma_{sl3}$ is defined as follows: $\g_p=\bigoplus_{p=-2}^2\mathfrak{sl}(3)_p$ with the grading defined by the Borel subalgebra consisting of the upper triangular matrices.
The representation is given by the adjoint representation. Identifying $(V, \phi)$ with $\gr V$ and shifting degree, we set $V=\bigoplus_{q=0}^{4}V_q$  $V_q=\g_{q-2}$. Therefore, $\dim \g_{-2}=1, \dim \g_{-1}=2, \dim \g_{0}=2, \dim \g_{1}=2, \dim \g_{2}=1$ and $\g_-=\g_{-2}\oplus \g_{-1}$ is the Heisenberg Lie algebra of dimension 3.

We therefore consider an osculating map $\varphi\colon(M,\f) \to \Flag(V,\phi)$ of symbol $(\g_-,V)$ described above, so that $(M,\f)$
is a filtered manifold with the symbol $\g_-$, that is, a contact manifold, and $(V,\phi)$ is the 8-dimensional vector space. (We shall later specify $L\subset GL(V)$). 

Thus this is the first simplest case where the source manifolds a non-trivial filtered manifolds. Note that if $\phi'$ is a subfiltration of $\phi$ then there is a canonical projection $\pi\colon \Flag(V,\phi) \to \Flag (V,\phi')$, in particular  by taking the last 1-dimensional subspace we have the canonical projection $\pi\colon \Flag(V,\phi) \to P(V)$. Composing with this projection, we have an immersion $\bar{\varphi}=\pi\circ \varphi\colon (M,\f) \to P(V)$ the correspondence $\varphi$ to $\bar {\varphi}$ being reciprocal, to study the osculating maps to the flag variety is just to study the immersions from 3 -dimensional contact manifolds to 7-dimensional projective space having osculating series of the above type. Thus this study may be viewed also as a contact generalization of the classical surface theory in projective spaces.

The linear differential equations corresponding an osculating map of the present type has the following form
\begin{equation*} 
\left\{ \begin{aligned}
X^2u &=a_1Xu+b_1Yu+c_1u,\\
Y^2u &=a_2Xu+b_2Yu+c_2u
\end{aligned}
\right.
\end{equation*}
where $X$ and $Y$ are two vector fields on a 3-dimensional contact manifold $(M,\f)$ spanning the contact distribution $\f^{-1}$, and $u\in\C^\infty(M)$.

This is a system of differential equations of weighted order to on a contact manifold and has 8-dimensional solution space 
if the coefficients satisfy all the compatibility conditions that we assume. 

Our general theory tells how to find the invariants of the above osculating maps and differential equations. On the basis of it we classify all the transitive osculating maps and differential equations of the type above with transitive symmetry algebra.

Classification of homogeneous submanifolds in a given homogeneous space $L/L^0$ has a long history going back to the first works of Sophus Lie on the classification of homogeneous curves in $P^2$. We reference Doubrov--Komrakov~\cite{DK} for general classification techniques and numerous works devoted to the classification of homogeneous surfaces in $A^3$ and $P^3$, as well as homogeneous submanifolds in $A^4$ and $P^4$, whose symmetry group has a non-trivial stabilizer.

Thus the present work may be viewed as a contact generalization of those preceding works but the general method developed 
in~\cite{DMM1} under which we work in this paper will also give light on the earlier works.

Surprisingly, similarity between the geometry of surfaces in $P^3$ and the geometry of osculating embeddings of $\mathfrak{sl}(3)$ type is observed in the classification results of submanifolds with transitive symmetry algebra. In particular, in both cases there exists a unique (up to equivalence) object with a symmetry algebra of submaximal dimension. In case of surfaces in $P^3$ this is the ruled Cayley cubic~\cite{cayley}, in our case this is an embedding with 4-dimensional symmetry algebra described in Subsection~\ref{ss_stab}. We call it a \emph{contact Cayley surface} by the analogy. Both have transitive symmetry with 1-dimensional stabilizer at a generic point.     

In the forthcoming paper we will give a general principle how to determine the transitive structures in general extrinsic geometry.
It is reduced to a purely algebraic problems which can be solves by explicit computations theoretically.
We may say that in our case of $\mathfrak{sl}(3)$ type  it is just a good example where we can carry out all the computations by hand with the aid of nice symmetry of the star David (see Section~\ref{sec:letoile}).

Now let us briefly describe the contents of this paper. We first recall the algebraic nature of the geometry that we are going to study since the Killing form of $\mathfrak {sl}(3)$ is left invariant by the adjoint representation we have
\[
	\mathfrak{sl}(3)\subset\mathfrak{so}(5,3)\subset \mathfrak{gl}(8).
\]
We then recompose it into $\fsl(3)$-irreducible components 
\begin{align*}
	\mathfrak {gl}(8)&=\Gamma_{0,0}+2\Gamma_{1,1}+\Gamma_{3,0} +\Gamma_{0,3} +\Gamma_{2,2},\\
	\mathfrak {so}(5,3)&=\Gamma_{1,1}+\Gamma_{3,0} +\Gamma_{0,3},
\end{align*}
where $\Gamma_{1,1}=\fsl(3)$.

By computing after Kostant~\cite{kostant} we have 
\[
	H^1_+(\g_-,\Gamma)=0
\]
for $\Gamma=\Gamma_{0,0}$, $\Gamma_{1,1}$, $\Gamma_{2,2}$ and 
\begin{align*}	
	H^1_+(\g_-,\Gamma_{3,0}) &= H^1_1(\g_-,\Gamma_{3,0})=\langle \xi^R_1 \rangle,\\
	H^1_+(\g_-,\Gamma_{0,3}) &= H^1_1(\g_-,\Gamma_{0,3})=\langle \xi^S_1 \rangle.
\end{align*}

The vanishing of the above first three cohomology groups implies the extrinsic geometry under the transformation group $GL(8)$ reduce to 
that under $SO(5,3)$, so that we study the extrinsic geometry by setting $L=SO(5,3)$.

Our starting  point is the Cartan bundle $Q\to M $ with structure group $G^0$ endowed with an $\mathfrak l$-valued 1-form $\omega$
satisfying 
\[
	d\omega+\tfrac12[\omega\wedge \omega]=0.
\]
According to the direct sum decomposition 
\[
	\lf=\g \oplus \g^{\perp},\ \g^{\perp}=\Gamma_{3,0}+\Gamma_{0,3}.
\]
We can write 
\[
	\omega=\omega_I +\omega_{II},\quad \omega_{II}=\chi\omega_I
\] and 
\[
	\chi=\chi^R+\chi^S,\quad \chi^R=\sum_{i=1}^6\chi^R_i,\quad \chi^S=\sum_{i=1}^6\chi^S_i.
\]
On account of the cohomology group obtained as above we can write 
\[
	\chi^R_1=h^R_1\xi^R_1,\quad \chi^S_1=h^S_1\xi^S_1.
\]  

Since the relevant cohomology group $H^1_r$ vanishes for $r>1$, according to the above theory the structure function $\chi_r$ is uniquely determined from $\chi_1$ inductively for $r>1$. A direct but cumbersome calculation then reveals unexpectedly simple exact formulas for $\chi_r$ as stated in Proposition~\ref{hRS}, which then leads us to much longer calculations to determine all transitive structures in our category.  

Now our main result can be formulated as follows.
\begin{thm}
	Let $\varphi\colon (M,\f)\to L/\phi^0 L\subset \Flag(V,\phi)$ be an osculating embedding of type $\fsl_3$ with a locally transitive symmetry algebra. Then, up to the action of $L$, it corresponds to one of the following systems of PDEs:\\[5mm]
	\begin{tabular}{|c|l|c|}
	\hline
	& \textbf{Equation} & \textbf{Symmetry algebra} \\
	\hline
	(O) & $Z_1^2u=Z_2^2u=0$ & $\sll(3,\R)$ \\
	\hline
	($I_0$) & $Z_1^2u=0,Z_2^2u=6Z_1u$ & 4-dim solvable \\
	\hline 
	($I_1$) & $Z_1^2u=0$, & 3-dim solvable \\
	& $Z_2^2u=6Z_1u+2P_2Z_2u-\big(\tfrac{24P_2^2}{25}\pm 1\big)u$ & \\[2mm]
	\hline 
	($I_2$) & $Z_1(Z_1\pm 2)u=0$, & 3-dim solvable \\
	& $Z_2^2u=6Z_1u\pm 9u$ & \\
	\hline
	($II_0$) & $Z_1^2u=-6Z_2u,Z_2^2u=6Z_1u$ & $\sll(2,\R)$ \\
	\hline 
	($II_1$) & $(Z_1-P_1)^2u=-6(Z_2-P_2)u+(P_1^2+3P_2)u$, & 3-dim solvable \\
	     & $(Z_2-P_2)^2u=6(Z_1-P_1)u+(P_2^2-3P_1)u$, & \\[2mm]
	     & $P_1P_2=-9$ & \\
	\hline
	($II_2$) & $(Z_1-P_1)^2u=-6(Z_2-P_2)u+(\tfrac{1}{4}P_1^2+3P_2)u$, & 3-dim solvable \\
	& $(Z_2-P_2)^2u=6(Z_1-P_1)u+(\tfrac{1}{4}P_2^2-3P_1)u$, & \\[3mm]
	& $P_1P_2=-144$ & \\
	\hline
	\end{tabular}
	
	\smallskip\noindent
	Here $Z_1,Z_2$ are certain left-invariant vector fields on a 3-dimensional real Lie group spanning a contact distribution, $u$ is an unknown smooth function on this group and $P_1,P_2$ are real constants. 
\end{thm}
More details on the symmetry algebras and their realization as subalgebras in $\fso(5,3)$ is summarized in Section~\ref{s:summary}. The proof of this theorem will be given in Sections~4, 5, 6 of this paper.

Part of these computations was done using Maple software. The first half of the classification (case (II), computations of Section~\ref{sec:nonvanishing}) was done initially without the use of computer software and then verified in Maple. The second half of the classification (case (I), computations of Section~\ref{s:vanishing})) was mostly done in Maple. The corresponding Maple worksheets can be found at \url{https://arxiv.org/abs/2308.06169}.   

\medskip
\textbf{Acknowledgments}. We would like to thank Yoshinori Machida and Dennis The for stimulating discussions on the topic of the paper.

\section{Extrinsic geometry of $\fsl_3$-type}

Let $\g=\fsl(3,\R)$ be endowed with a grading $\g=\oplus_{p\in\Z}\g_p$ determined by the Borel subalgebra of $\fsl(3,\R)$, also known as \emph{contact grading}. Let $V=\oplus V_q$ be the graded $\g$-module given by $V_q=\g_{q-2}$, the representation of $\g$ on $V$ being the adjoint representation. Let $\kappa$ be the Killing form of $V$ ($=\fsl(3,\R)$).  Since $\ad X$ ($X\in\g$) preserves $\kappa$, we have an embedding
\[
\ad\colon \g \to \of(V,\kappa)\subset \gl(V).
\]
Note that here $\of(V,\kappa)=\of(3,5)$. Note also that the grading of $V$ induces that of $\gl(V)$ and then that of $\of(V,\kappa)$, so that we have
\begin{align*}
	&\g_p\to\of(V,\kappa)_p\subset \gl(V)_p,\\
	& \of(V,\kappa)=\oplus \of(V,\kappa)_p,
	& \gl(V)=\oplus \gl(V)_p.
\end{align*}

We denote by $\phi$ the filtrations induced by these gradings:
\[
	\phi^p\g = \oplus_{i\ge p}\g_i,\quad \phi^p V=\oplus_{i\ge p} V_p,
\]
and so on. We denote by $G$ the group $SL(3,\R)$ and by $L$ the group $O(V,\kappa)$ and $\lf$ its Lie algebra. 

We study the extrinsic geometries in $\Flag(V,\phi)$. In particular, we consider now extrinsic geometries in $L/\phi^0L\subset \Flag(V,\phi)$.

From~\cite{DMM1} we have
\begin{thmm}\label{thm1} To each osculating map 
	\[
	 \varphi\colon (M,\f)\to L/\phi^0L
	\]
of type $(\g_{-},V; L)$ there canonically corresponds the pair $(P,\omega)$, where 
\begin{enumerate}
	\item $P$ is a principal frame bundle over $M$ with the structure group $G^0=\phi^0 G$;
	\item $\omega$ is an $\lf$-valued 1-form satisfying
	\begin{enumerate}
		\item[i)]  $\langle \tilde A, \omega\rangle=A$, $A\in \g^0$;
		\item[ii)] $R_a^*\omega=\Ad (a^{-1})\omega$, $a\in G^0$; 
		\item[iii)] $L_{\tilde A}\omega = -\ad (A)\omega$, $A\in\g^0$;
		\item[(iv)] $d\omega+\frac{1}{2}[\omega,\omega]=0$;
		\item[v)] if we decompose $\omega$ as
		\[
			\omega=\omega_{I}+\omega_{II}
		\]
		according to the direct sum decomposition
		\[
			\lf = \g\oplus \g^{\perp},
		\]
		then $\omega_{I}\colon T_zP\to \g$ is a linear isomorphism for any $z\in P$;
		\item[vi)] if we write $\omega_{II}=\chi \omega_{I}$, then $\chi$ is a $\Hom(\g_{-},\g^{\perp})$-valued function on $P$ and
		\[
			\partial^*\chi = 0.
		\]
		Moreover, $\chi_j=0$ for $j\le 0$, where 
		\[
			\chi = \sum_j \chi_j,\quad \chi_j\colon P\to \Hom(\g_{-},\g^{\perp}_j).
		\]
	\end{enumerate}
\end{enumerate}
\end{thmm}

\section{L'\'etoile de Davide}
\label{sec:letoile}

We have an irreducible decomposition of $\gl(V)=\gl(8,\R)$ as $\fsl(3,\R)$-module:
\[
\gl(8,\R)=\Gamma_{0,0}\oplus 2\Gamma_{1,1}\oplus \Gamma_{3,0}\oplus\Gamma_{0,3}\oplus\Gamma_{2,2},
\]
and an irreducible decomposition of $\of(V,\kappa)=\of(5,3)$:
\[
\of(5,3)=\Gamma_{1,1}\oplus \Gamma_{3,0}\oplus\Gamma_{0,3},
\]
where $\Gamma_{a,b}$ denotes the irreducible module of highest weight $a\lambda_1+b\lambda_2$, $\{\lambda_1, \lambda_2\}$ being the fundamental system of weights. Note that $\Gamma_{1,1}\equiv \fsl(3,\R)$. Thus,
\[
\g^\perp = \Gamma_{3,0}\oplus\Gamma_{0,3}.
\] 
Note also that both $\Gamma_{3,0}$ and $\Gamma_{0,3}$ satisfy the bracket relations:
\begin{align*}
	&[\Gamma_{3,0}, \Gamma_{3,0}]\subset \Gamma_{0,3},\\
	&[\Gamma_{0,3}, \Gamma_{0,3}]\subset \Gamma_{3,0},\\
	&[\Gamma_{3,0},\Gamma_{0,3}]\subset \Gamma_{1,1} = \fsl(3,\R).
\end{align*}
 This directly follows from the decomposition of $\fsl(3,\R)$-modules:
\begin{align*}
	\Gamma_{3,0}\otimes \Gamma_{0,3} &= \Gamma_{3,3} \oplus \Gamma_{2,2} \oplus\Gamma_{1,1}\oplus \Gamma_{0,0},\\
    \wedge^2 \Gamma_{3,0} &= \Gamma_{4,1}\oplus \Gamma_{0,3}
\end{align*}
and similarly for $\wedge^2 \Gamma(0,3)$.

Let us fix the notation for the representations of $\fsl(3,\R)$. We take the basis $\{A_1,A_2,\dots,A_8\}$ of $\fsl(3,\R)$ as:
\begin{equation}\label{eq:basis}
\begin{aligned}
A_1 &= E_{13}, A_2=E_{12}, A_3=E_{23},\\
A_4 &=H_1=E_{11}-E_{22}, A_5=H_2=E_{22}-E_{33},\\
A_6 &= E_{32}, A_7=E_{21}, A_8=E_{31},
\end{aligned}
\end{equation}
where $E_{ij}$ denotes the $(i,j)$-matrix element in $\gl(3,\R)$. For $X\in\fsl(3,\R)$, elements $X$, $\ad X$ and their matrix representations with respect to the above basis are all identified and denoted simply by $X$. Thus, for example,
\[
 E_{13} = \ad E_{13} = \begin{pmatrix} 
 0 & 0 & 0 & -1 & -1 & 0 & 0 & 0 \\
 0 & 0 & 0 & 0  & 0  & 1 & 0 & 0 \\
 0 & 0 & 0 & 0  & 0  & 0 & -1 & 0 \\
 0 & 0 & 0 & 0  & 0  & 0 & 0 & 1 \\
 0 & 0 & 0 & 0  & 0  & 0 & 0 & 1 \\
 0 & 0 & 0 & 0  & 0  & 0 & 0 & 0 \\ 
  0 & 0 & 0 & 0  & 0  & 0 & 0 & 0 \\ 
   0 & 0 & 0 & 0  & 0  & 0 & 0 & 0
  \end{pmatrix}. 
\] 
We shall often write $e_0, e_1, e_2$ and $\check{e}_0, \check{e}_1, \check{e}_2$ for the elements $A_8, A_7, A_6$ and $A_1, A_2, A_3$ respectively. The dual basis of $\{A_1,A_2,\dots, A_8\}$ is denoted $\{A_1^*, A_2^*, \dots, A_8^*\}$.

The weight diagram of the $\fsl(3,\R)$-module $\lf=\of(5,3)=\Gamma_{1,1}\oplus \Gamma_{3,0}\oplus \Gamma_{0,3}$ forms l'\'etoile de Davide:
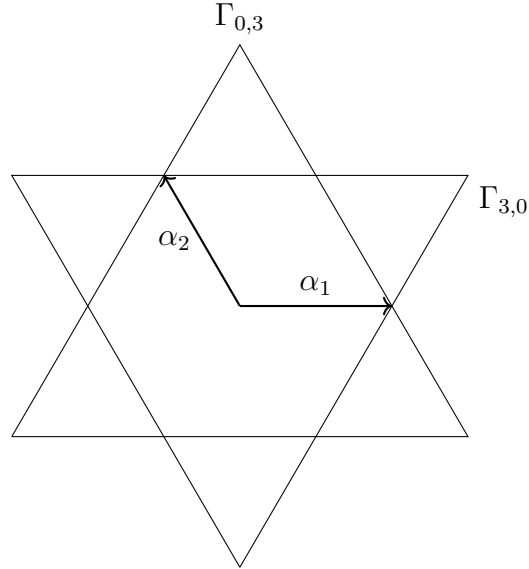
\begin{center}
\begin{tikzpicture}
	\draw (-3, -1.732) -- (0, 3.464) -- (3, -1.732) -- (-3, -1.732);
	\draw (0, -3.464) -- (-3, 1.732) -- (3, 1.732) -- (0, -3.464);
	\draw [thick, ->] (0,0) -- (2,0);
	\draw [thick, ->] (0,0) -- (-1, 1.732);
	\node[above] at (0, 3.464) {$\Gamma_{0,3}$};
	\node[below right] at (3, 1.732) {$\Gamma_{3,0}$};
	\node[above] at (1,0) {$\alpha_1$};
	\node[left] at (-0.5, 0.866) {$\alpha_2$};
\end{tikzpicture}	
\end{center}

Note that the central hexagon forms the root diagram of $\g$. For later use we fix weight vectors for $\Gamma_{3,0}$ and $\Gamma_{0,3}$. Put
\[
 R_{2,1} = A_1\otimes A_7^*-A_2\otimes A_8^*
\]
and then
\begin{align*}
R_{1,1} &= [e_1, R_{2,1}], &
R_{1,0} &= [e_2, R_{1,1}], &
R_{0,1} &= [e_1, R_{1,1}],\\
R_{0,0} &= [e_2, R_{0,1}], &
R_{0,-1} &= [e_2, R_{0,0}],&
R_{-1,1} &= [e_1, R_{0,1}],\\
R_{-1,0} &= [e_2, R_{-1,1}],&
R_{-1,-1} &= [e_2, R_{-1,0}],&
R_{-1,-2} &= [e_2, R_{-1,-1}].
\end{align*}

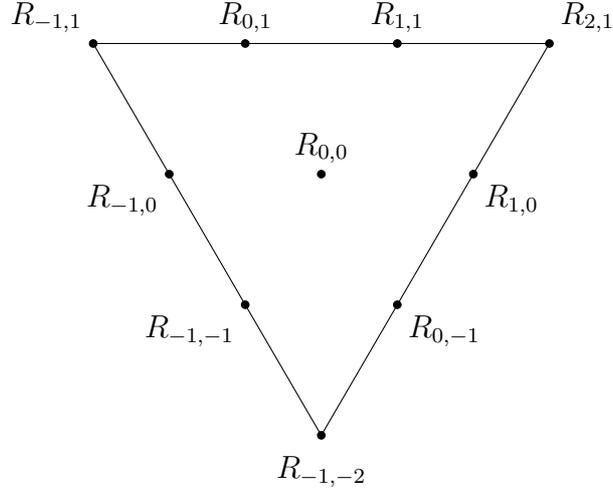
\begin{center}
\begin{tikzpicture}
	\draw (0, -3.464) -- (-3, 1.732) -- (3, 1.732) -- (0, -3.464);
	\draw[fill] (0,0) circle [radius=0.05];
	\draw[fill] (2,0) circle [radius=0.05];
	\draw[fill] (1,-1.732) circle [radius=0.05];
	\draw[fill] (-2,0) circle [radius=0.05];
	\draw[fill] (-1,-1.732) circle [radius=0.05];
	\draw[fill] (1,1.732) circle [radius=0.05];
	\draw[fill] (-1,1.732) circle [radius=0.05];
	\draw[fill] (0, -3.464) circle [radius=0.05];
	\draw[fill] (-3, 1.732) circle [radius=0.05];
	\draw[fill] (3, 1.732) circle [radius=0.05];
	\node[below] at (0, -3.6) {$R_{-1,-2}$};
	\node[below right] at (1, -1.732) {$R_{0,-1}$};
	\node[below right] at (2, 0) {$R_{1,0}$};
	\node[above right] at (3, 1.732) {$R_{2,1}$};
	\node[above] at (1, 1.732) {$R_{1,1}$};
	\node[above] at (-1, 1.732) {$R_{0,1}$};
	\node[above left] at (-3, 1.732) {$R_{-1,1}$};
	\node[below left] at (-2, 0) {$R_{-1,0}$};
	\node[below left] at (-1, -1.732) {$R_{-1,-1}$};
	\node[above] at (0, 0) {$R_{0,0}$};
\end{tikzpicture}
\end{center}

Similarly for $\Gamma_{0,3}$ we put
\[
 S_{1,2} = A_1\otimes A_6^*-A_3\otimes A_8^*
\]
and 
\begin{align*}
	S_{1,1} &= [e_2, S_{1,2}], &
	S_{0,1} &= [e_1, S_{1,1}], &
	S_{1,0} &= [e_2, S_{1,1}],\\
	S_{0,0} &= [e_1, S_{1,0}], &
	S_{-1,0} &= [e_1, S_{0,0}],&
	S_{1,-1} &= [e_2, S_{1,0}],\\
	S_{0,-1} &= [e_1, S_{1,-1}],&
	S_{-1,-1} &= [e_1, S_{0,-1}],&
	S_{-2,-1} &= [e_1, S_{-1,-1}].
\end{align*}
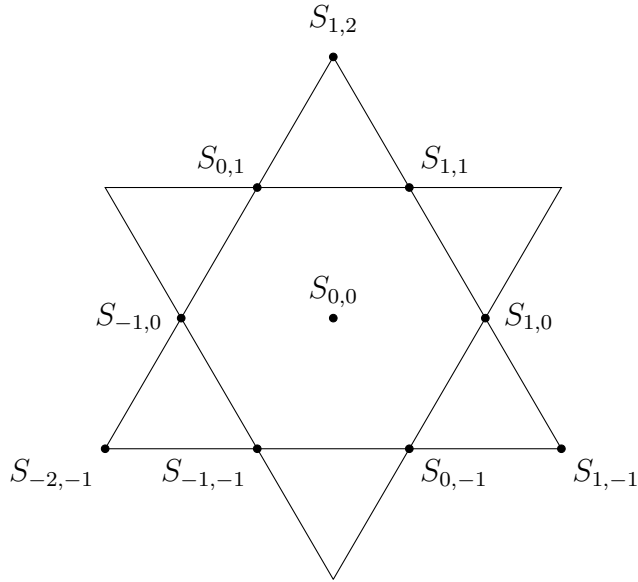
\begin{center}
\begin{tikzpicture}
	\draw (-3, -1.732) -- (0, 3.464) -- (3, -1.732) -- (-3, -1.732);
	\draw (0, -3.464) -- (-3, 1.732) -- (3, 1.732) -- (0, -3.464);
	\draw[fill] (0,0) circle [radius=0.05];
	\draw[fill] (2,0) circle [radius=0.05];
	\draw[fill] (1,-1.732) circle [radius=0.05];
	\draw[fill] (-2,0) circle [radius=0.05];
	\draw[fill] (-1,-1.732) circle [radius=0.05];
	\draw[fill] (1,1.732) circle [radius=0.05];
	\draw[fill] (-1,1.732) circle [radius=0.05];
	\draw[fill] (0, 3.464) circle [radius=0.05];
	\draw[fill] (-3, -1.732) circle [radius=0.05];
	\draw[fill] (3, -1.732) circle [radius=0.05];
	\node[above] at (0, 3.6) {$S_{1,2}$};
	\node[below right] at (1, -1.732) {$S_{0,-1}$};
	\node[right] at (2.1, 0) {$S_{1,0}$};
	\node[below left] at (-3, -1.732) {$S_{-2,-1}$};
	\node[above right] at (1, 1.732) {$S_{1,1}$};
	\node[above left] at (-1, 1.732) {$S_{0,1}$};
	\node[below right] at (3, -1.732) {$S_{1,-1}$};
	\node[left] at (-2.1, 0) {$S_{-1,0}$};
	\node[below left] at (-1, -1.732) {$S_{-1,-1}$};
	\node[above] at (0, 0) {$S_{0,0}$};
\end{tikzpicture}	
\end{center}

\section{The structure function $\chi$}
\subsection{Structure equations}
From 
\[
d\omega + \frac12[\omega,\omega]=0
\]
and 
\[
 \omega = \omega_{I}+\omega_{II},\quad \omega_{II} = \chi\omega_{I}
\]
and from normality condition we have
\begin{align*}
 \partial\chi & = -\cD \wedge \chi - \tfrac12[\chi\wedge \chi]_{II}+\tfrac12\chi([\chi\wedge\chi]_{-}),\\
 \partial^*\chi &= 0, 
\end{align*}
where $\partial$ denotes the cohomology codifferential 
\[
	\partial\colon C^k(\g_{-},\g^{\perp})\to C^{k+1}(\g_{-},\g^{\perp}),
\]
$\partial^*$ denotes its dual 
\[
\partial^*\colon C^{k+1}(\g_{-},\g^{\perp})\to C^k(\g_{-},\g^{\perp}),
\]
and
\begin{align*}
	(\cD\wedge \chi)(u,v) &= L_{\tilde u}(\chi)(v)- L_{\tilde v}(\chi)(u),\\
	[\Phi\wedge \Psi](u,v) &= [\Phi(u),\Psi(v)]-[\Phi(v),\Psi(u)],
\end{align*}
for any $u,v\in\g_{-}$.

At each degree $p\ge 1$ we have
\begin{equation}\label{chi_p}
\begin{aligned}
\partial\chi_p &=- \sum_{i=1,2}\cD_i\wedge \chi_{p-i}-\frac12 \sum_{\substack{i+j=p \\ i,j\ge 1}}[\chi_i,\chi_j] + \frac12\sum_{\substack{i+j+k=p \\ i,j,k\ge 1}}\chi_i([\chi_j,\chi_k]_{-}),
\\
\partial^*\chi_p &=0,
\end{aligned}
\end{equation}
which determines $\chi_p$ inductively, and uniquely up to the cohomology group $H^1_{+}(\g_{-},\g^{\perp})$. 

In our case we can determine $\chi_p$ completely as follows: 

Computing cohomology via Kostant theorem~\cite{kostant}, we have:
\begin{prop}
	\[
		H_{+}^1(\g_{-},\g^{\perp})= \langle \xi_1^R\rangle \oplus \langle \xi_1^S\rangle,
	\]
	where 
	\[
	\begin{cases*}
	\xi_1^R = R_{-1,1}\otimes e_2^*=-6(A_3\otimes A_2^* - A_7\otimes A_6^*)\otimes e_2^*,\\
	\xi_1^S = S_{1,-1}\otimes e_1^*=6(A_2\otimes A_3^*-A_6\otimes A_7^*)\otimes e_1^*.
	\end{cases*}
	\]
\end{prop}
Thus, we can write 
\[
	\chi=\chi^R+\chi^S,\quad \chi^R=\sum_{i\ge 1}\chi^R_i,\quad \chi^S=\sum_{j\ge 1}\chi^S_j,
\]
and 
\[
	\chi^R_1 = h_1^R\xi_1^R,\quad \chi^S_1 = h_1^S\xi_1^S,
\]
where $h_1^R$ and $h_1^S$ are some functions on $P$.

Playing with l'\'etoile de Davide and carrying detailed computations, we can explicitly solve equations~\eqref{chi_p} for $\chi$.
\begin{prop}\label{hRS} The function $\chi$ has the form:
\begin{align*}
	\chi_1^R &= h_1^R R_{-1,1}\otimes e_2^* &&= h_1^R \xi_1^R,\\
	\chi_2^R &= D_{e_1}h_1^R \big(-\tfrac34 R_{0,1}\otimes e_2^* -\tfrac14 R_{-1,1}\otimes e_0^*\big) && = -\tfrac14 D_{e_1}h_1^R \rho(\check e_1) \xi_1^R, \\
	\chi_3^R &= D^2_{e_1} h_1^R \big(\tfrac12 R_{1,1}\otimes e_2^* + \tfrac14 R_{0,1}\otimes e_0^*\big) && = \tfrac{1}{24} D^2_{e_1}h_1^R \rho(\check e_1)^2\xi_1^R,\\
	\chi_4^R &= D^3_{e_1}h_1^R \big(-\tfrac14 R_{2,1}\otimes e_2^* -\tfrac14 R_{1,1}\otimes e_0^*\big) && = -\tfrac{1}{36\cdot 4}D^3_{e_1}h_1^R\rho(\check e_1)^3\xi_1^R,\\
	\chi_5^R &= D^4_{e_1}h_1^R \big(\tfrac14 R_{2,1}\otimes e_0^*\big) && = \tfrac{1}{36\cdot 4\cdot 4}D^4_{e_1}h_1^R \rho(\check e_1)^4 \xi_1^R. 
\end{align*}
and
\begin{align*}
	\chi_1^S &= h_1^S S_{1,-1}\otimes e_1^* &&= h_1^S \xi_1^S,\\
	\chi_2^S &= D_{e_2}h_1^S \big(-\tfrac34 S_{1,0}\otimes e_1^* + \tfrac14 S_{1,-1}\otimes e_0^*\big) && = -\tfrac14 D_{e_2}h_1^S \rho(\check e_2) \xi_1^S, \\
	\chi_3^S &= D^2_{e_2} h_1^S \big(\tfrac12 S_{1,1}\otimes e_1^* - \tfrac14 S_{1,0}e_0^*\big) && = \tfrac{1}{24} D^2_{e_2}h_1^S \rho(\check e_2)^2\xi_1^S,\\
	\chi_4^S &= D^3_{e_2}h_1^S \big(-\tfrac14 S_{1,2}\otimes e_1^* + \tfrac14 S_{1,1}\otimes e_0^*\big) && = -\tfrac{1}{36\cdot 4}D^3_{e_2}h_1^S\rho(\check e_2)^3\xi_1^S,\\
	\chi_5^S &= D^4_{e_2}h_1^S \big(-\tfrac14 S_{1,2}\otimes e_0^*\big) && = \tfrac{1}{36\cdot 4\cdot 4}D^4_{e_2}h_1^S \rho(\check e_2)^4 \xi_1^S. 
\end{align*}
\end{prop}
\begin{prop}
	The functions $\xi_1^R$ and $\xi_1^S$ satisfy the following compatibility conditions:
	\[
	\begin{cases}
		\big(D_{e_0}+\tfrac14 D_{e_2}D_{e_1}\big) h_1^R = 0,\\
		\big(D_{e_0}-\tfrac14 D_{e_1}D_{e_2}\big) h_1^S = 0,
	\end{cases}
	\]
	and
	\[
	\begin{cases}
		D_{e_2}^5 h_1^S = -6 \big(D^4_{e_1}h_1^R\big)h_1^R + 6\big(D^3_{e_1}h_1^R\big)\big(D_{e_1}h_1^R\big)-3\big(D^2_{e_1}h_1^R\big)^2,\\
		D_{e_1}^5 h_1^R = \phantom{-}6\big(D^4_{e_2}h_1^S\big)h_1^S - 6\big(D^3_{e_2}h_1^S\big)\big(D_{e_2}h_1^S\big)+3\big(D^2_{e_2}h_1^S\big)^2.
	\end{cases}
\]
\end{prop}

\subsection{Geometric interpretation of vanishing $\chi^R$ or $\chi^S$}

Consider the osculating map:
\[
\varphi\colon (M^3,\f)\to \Flag(V,\phi), \quad V=\R^8\cong \fsl(3,\R),
\]
modeled by the highest root orbit of the adjoint representation of $SL(3,\R)$.

Note that this map defines also the embedding $M^3\to P^7=P(V)$ via the natural projection $\Flag(V,\phi)\to P^7$. In fact, $\varphi$ can be reconstructed from this embedding via the flag of (weighted) osculating spaces. 

According to Theorem~\ref{thm1} we have the natural extrinsic normal Cartan connection $(P,\omega)$ on $M^3$:
\[
\pi\colon P\to M^3,\quad TP\to \of(5,3).
\]
Recall that $\omega$ is decomposed as $\omega=\omega_{I}+\omega_{II}$ according to the decomposition $\of(5,3)=\g\oplus\g^{\perp}$, where $\g$ is an image of the adjoint representation of $\fsl(3,\R)$. 

Then $\om_{I}$ defines an (intrinsic) Cartan connection on $M^3$ modeled by the homogeneous space $PSL(3,\R)/B$, where $B$ is the Borel subgroup in $PSL(3,\R)$ consisting of upper triangular matrices. 
In particular, this defines the splitting of $\f^{-1}$ into two line bundles $l^R\oplus l^S$, where 
\begin{align*}
	l^R &= \pi_*\om_{I}^{-1}(e_2),\quad e_2 = E_{32};\\
	l^S &= \pi_*\om_{I}^{-1}(e_1),\quad e_1 = E_{21}.
\end{align*}

In accordance with the terminology used for hyperbolic surfaces in $P^3$, we call these line bundles \emph{asymptotic directions} of the osculating map $\varphi$, and their integral curves the \emph{asymptotic curves}. To distinguish between these two line bundles we call integral curves of $l^R$ the $R$-asymptotic curves or just $R$-curves and similarly $S$-asymptotic curves or $S$ curves for integral curves of $l^S$.

\begin{prop} $R$-asymptotic curves ($S$-asymptotic curves) embed into $P^7$ as straight lines if and only if $\chi^R\equiv0$ (resp. $\chi^S\equiv0$). 
\end{prop}
\begin{proof}
	Let $e_2^*$ be the fundamental vector field on $P$ corresponding to $e_2\in\fsl(3,\R)$ with respect to the Cartan connection $\om_I$, that is $\om_I(e_2^*)=e_2$. Then $\om(e_2^*)$ defines a \emph{moving frame} over each $R$-asymptotic curve. 
	
	The value of $\om(e_2^*)$ in basis~\eqref{eq:basis} has the form:
	\[
		\om(e_2^*) = \begin{pmatrix}
			* & * & * & * & * & * & * & * \\
			-1 & * & * & * & * & * & * & * \\
			0 & -6 h_1^R & * & * & * & * & * & * \\
			0 & 0 & 0 & 0 & 0 & * & * & * \\
			0 & 0 & -1 & 0 & 0 & * & * & * \\
			0 & 0 & 0 & -1 & 2 & * & * & * \\
			0 & 0 & 0 & 0 & 0 & 6h_1^R & * & * \\
			0 & 0 & 0 & 0 & 0 & 0 & 1 & * \\
		\end{pmatrix}
	\]
	It follows that $\om(e_2^*)$ acts on $A_1=E_{13}$ (the highest weight vector of the adjoint representation of $PSL(3,\R)$) as:
	\begin{align*}
		\om(e_2^*)\colon E_{13}&\mapsto -E_{12}\mod \langle E_{13}\rangle,\\
		\om(e_2^*)\colon E_{12}&\mapsto -6 h^R_1 E_{23} \mod \langle E_{12}, E_{13}\rangle.
	\end{align*}
	Thus, the osculating flag of the $R$-asymptotic curves stabilizes at the 2-dimensional subspace $\langle E_{13}, E_{12}\rangle$ (or at an 1-dimensional line in $P^7$) if and only if $\chi^R_1$ vanishes identically on $P$. According to Proposition~\ref{hRS} this also implies that $\chi^R$ vanishes identically on $P$.
\end{proof}

On the other hand, we have:
\begin{prop}
	If either $\chi^R$ or $\chi^S$ vanishes identically, then the Cartan connection defined by $\om_I$ is flat. 
\end{prop}
\begin{proof}
	Indeed, assuming that $\chi_R=0$, we get that $\om_{II}$ takes values only in the representation $\Gamma_{0,3}$ of the $\fsl(3,\R)$ decomposition of $\bar\g^{\perp}$ as $\Gamma_{3,0}\oplus \Gamma_{0,3}$. Note that $[\Gamma_{0,3}, \Gamma_{0,3}]\subset \Gamma_{3,0}$. We see that in the decomposition
	\[
	[\om,\om]=[\om_I+\om_{II},\om_{I}+\om_{II}]
	\]
	only the term $[\om_{I},\om_{I}]$ lies in $\bar\g$. Hence, the structure equation $d\om=\tfrac{1}{2}[\om,\om]$ implies that
	\[
		d\om_I + \tfrac{1}{2}[\om_I,\om_I] = 0,
	\]
	and the Cartan connection $\om_I$ has zero curvature.
\end{proof}

\section{Transitive structures}
\subsection{Transitive Cartan bundles and transitive embeddings}
Let $\varphi\colon (M,\f)\to \Flag(V,\phi)$ be an osculating embedding of type $(\g_{-},V,L)$ and let $(P,\omega)$ be the corresponding Cartan bundle with the canonical projection $\pi\colon P\to M$. 

\begin{df}
We say that $(P,\omega)$ is \emph{transitive}, if the group $\Aut(P,\omega)$ acts transitively on the fibers of the projection $\pi$. In other words, for any two points $x_1,x_2\in M$ there exist $z_i\in\pi^{-1}(x_i)$, $i=1,2$, and an automorphism $\psi\in\Aut(P,\omega)$ such that $\psi(z_1)=z_2$.
\end{df}

We say that the embedding $\varphi\colon (M,\f)\to \Flag(V,\phi)$ is transitive, if the corresponding Cartan bundle is transitive.

For the classification, we consider that $(P,\omega)$ is a principal fibre bundle over a neighborhood of a point $\mathring{x}\in(M,\f)$ and that each fibre of $P$ is connected. 

Recall that we have the following commutative diagram:
\[
\begin{CD}
P  @>\Phi>>  L \\
@VVV @VVV \\
M  @>\varphi>>  L/L^0
\end{CD}
\]	
such that $\Phi$ is a bundle map with $\Phi^*\Omega_L =\omega$. The map $\Phi$ ia unique up to a left multiplication by $L_a$, $a\in L$. 

For any $h\in \Aut(P,\omega)$ we have 
\[
\begin{CD}
P  @>\Phi>>  L \\
@V{h}VV @VV{L_b}V \\
P  @>\Phi>>  L
\end{CD}
\]
for some unique $b\in L$, which determines an embedding
\[
\iota\colon \Aut(P,\omega)\to L.
\]
Different choice $\Phi'$ gives a conjugate embedding $\iota'$. 

Fix a point $\mathring{z}\in P$ such that $\pi(\mathring{z})=\mathring{x}$. Then there exists a unique embedding $\Phi\colon P\to L$ such that $\Phi(\mathring{z})=e_L$ and $\Phi^*\Omega_L = \omega$. This $\Phi$, in its turn, determines an embedding
\[
\iota \colon \Aut(P,\omega)\to L
\]
and, hence, the injective map of Lie algebras
\[
	\iota_*\colon \aut(P,\omega)\to \lf.
\]
This Lie algebra homomorphism is given by:
\[
\omega_{\mathring{z}}\colon T_{\mathring{z}}(\Aut(P,\omega)\mathring{z})\to \lf,
\]
where by $\Aut(P,\omega)\mathring{z}$ we denote the orbit of $\Aut(P,\omega)$ through the point~$\mathring{z}$.

Let $H$ be the automorphism group $\Aut(P,\omega)$ and let $\Hf$ be the corresponding Lie algebra. Let $Q$ be the $H$-orbit through $\mathring{z}\in P$. The tangent space $T_{\mathring{z}}Q$ is identified with $\Hf$, and we have a Lie algebra embedding:
\[
\omega_{\mathring{z}}\colon \Hf \to \lf.
\] 
The filtration of $T_{\mathring{z}}P$ induces that of $\Hf$, and $	\omega_{\mathring{z}}$ preserves the filtrations, that is,
\[
\omega_{\mathring{z}}(\phi^k\Hf)\subset \phi^k\lf.
\]
Then we have an embedding of graded Lie algebras:
\[
\hf = \gr \Hf \to \gr\lf = \lf. 
\]
By the assumption of $H$ being 	``base transitive'', we have
\[
\hf_{-} = \g_{-}.
\]

Now we are going to look for $\mathring{z}\in P$, which gives a normal form of $(P,\omega)$.  

Define now the function $\chi\colon \g_{-}\to \g^{\perp}$ via the following commutative diagram:
\[
\begin{tikzcd}
	T_{\mathring{z}}Q\arrow[rr, "\omega_{II}"]\arrow[dr, "\omega_{-}"]& & \g^{\perp} \\ 	& \g_{-} \cong \g / \phi^0\g \arrow[ur, "\chi"]  & 
\end{tikzcd}
\]
Note that $\chi$ is well-defined as both $\om_{II}$ and $\om_{-}$ vanish on $\ker \pi_* \subset T_{\mathring{z}}Q$, and $\om_{-}$ induces an isomorphism of $T_{\mathring{z}}Q/\ker \pi_*$ and $\g_{-}$.

\subsection{The case of a non-trivial stabilizer}\label{ss_stab}
Let us describe all transitive embeddings $\varphi\colon (M,\f)\to \Flag(V,\phi)$ for which the group $H=\Aut(P,\omega)$ acts on $M$ with a stabilizer of dimension at least $1$, or, equivalently, when $\phi^0\Hf\ne 0$, or when $\hf_{0}\ne 0$.  

If both $\chi_1^R$ and $\chi_1^S$ vanish identically, then we know that the embedding $\varphi$ is flat and $\Hf=\g$. We exclude this trivial case from consideration and assume that at least one of $\chi_1^R$ and $\chi_1^S$ does not vanish identically.

Assume first that both $\chi_1^R$ and $\chi_1^S$ do not vanish. Then there exists a point $\mathring{z}\in P$ such that $h_1^R(\mathring{z})=h_1^S(\mathring{z})=1$, that is 
\begin{align*}
	\chi_1^R(\mathring{z}) &=\xi_1^R = R_{-1,1}\otimes e_2^*,\\
	\chi_1^S (\mathring{z}) & =\xi_1^S = S_{1,-1} \otimes e_1^*. 
\end{align*}
Indeed, since $R^*_a\chi = \rho(a)^{-1}\chi$ for $a\in G^0$, this can be realized by a translation of some $\bar a \in G^0/\chi^1 G$. 

Note that the Lie algebra 
\begin{equation}\label{eq:stabRS}
	\g_0^{R,S} = \{ x\in\g_0 \mid \rho(x)\xi_1^R=\rho(x)\xi_1^S = 0\}
\end{equation}
is trivial. Indeed, for $x=\lambda_1H_1 + \lambda_2H_2$, we have  
\begin{multline}\label{ann_chi}
\rho(x)\xi_1^S = \rho(x) S_{1,-1}\otimes e_1^* \\ 
= \langle 2\alpha_1-\alpha_2, \lambda_1H_1 + \lambda_2H_2\rangle \xi_1^R = (5\lambda_1 - 4\lambda_2)\xi_1^S,
\end{multline}
and similarly
\begin{equation}\label{ann_chi_R}
 \rho(x)\xi_1^R = (-4\lambda_1+5\lambda_2)\xi_1^R,
\end{equation}
which implies~\eqref{eq:stabRS}.

Thus, we have $\hf_{0} = \g^{R,S}_0 = 0$, and the stabilizer is necessarily trivial in this case.

Assume now that $\chi_1^S$ vanishes identically, and $\chi_1^R$ does not vanish. Choose a point $\mathring{z}\in P$ such that $h_1^R(\mathring{z})=1$, that is  
\begin{align*}
	\chi_1^R(\mathring{z}) &=\xi_1^R = R_{-1,1}\otimes e_2^*,\\
	\chi_1^S (\mathring{z}) & =\xi_1^S = 0. 
\end{align*}

Let us define
\begin{equation}\label{eq:stabS}
	\g_0^{R} = \{ x\in\g_0 \mid \rho(x)\xi_1^R = 0\}.
\end{equation}

According to~\eqref{ann_chi_R} we have
\[
\g_0^{R} = \langle 5 H_1 + 4 H_2 \rangle.
\] 
Note that $\hf_{0}\subset \g_0^{R}$ and thus $\hf\subset \Prol(\g_{-}\oplus \g_0^{R})$. It is easy to see that
\[
\Prol(\g_{-}\oplus \g_0^{R}) = \g_{-}\oplus \g_0^{R}.
\]
So, the dimension of $H$ is maximally $4$, and this dimension is achieved if and only if $\hf_{0}=\g_0^{S}$.

Let us assume that this is indeed the case. Then $\Hf^0$ contains a unique element $h=5 H_1 + 4 H_2 + h'$, where $h'\in \phi^1\g$. But since $5 H_1 + 4 H_2$ acts on $\phi^1\g$ with positive eigenvalues, we can always modify the point $\mathring{z}$ by the action of $\phi^1 G$ to get $h'=0$. So, without loss of generality we can assume that 
\[
\Hf^0 = \langle h=5 H_1 + 4 H_2 \rangle.
\]

\begin{prop}\label{p:cayley}  Assume that the embedding $\varphi\colon (M,\f)\to \Flag((V,\phi)$ is not flat and $\dim \Hf \ge 4$. Then up to the action of $G^{(0)}$  we have:
	\begin{itemize}
		\item $\chi = \xi^R_1$;
		\item $\Hf = \langle e_0, e_1, e_2+R_{-1,1}, 5 H_1 + 4 H_2 \rangle$.
	\end{itemize}
\end{prop}
\begin{proof}
	As it was shown above, if  $\dim \Hf^0 >0$ and $\xi^R_1\ne 0$, then up the action of $\phi^1G$  we can assume that $\Hf^0 = \langle 5 H_1 + 4 H_2 \rangle$.
	
	It is clear that $\chi$ is $\Hf^0$-invariant, that is $h.\chi=0$. This implies that $\chi$ belongs to weight subspaces of weights $k(\alpha_1-2\alpha_2)$, $k\in\mathbb{Q}$ in the weight decomposition of $\Hom(\g_{-}, \g^{\perp})$.
	
	Note that the vector space $\Hom(\g_{-}, \Gamma_{0,3})$ has the following weights with respect to the action of $\g_0$ (represented in terms of simple roots of $\g=\fsl(3,\R)$):
	\[
	\alpha + \omega, 
	\]
	where $\alpha=\alpha_1$, $\alpha_2$, or $\alpha_1+\alpha_2$ and $\omega$ is one of the weights of $\Gamma_{3,0}$. From l'\'Etoile diagram of $\Gamma_{0,3}$ we see that none of these weights is equal (or proportional) to $\alpha_1-2\alpha_2$. 
	
	Similarly, considering $\Hom(\g_{-}, \Gamma_{3,0})$ we see that the only weight proportional to $\alpha_1-2\alpha_2$ is $2\alpha_1-\alpha_2$ itself, and the corresponding weight subspace is spanned by $R_{-1,1}\otimes e_2^*=\xi^R_1$. 
	
	As we already know that $h^R_1(\mathring{z^0})=1$, we get $\chi = \xi^R_1$. This implies that $\Hf$ can be defined by an $\Hf^0$-invariant map $\psi\in\Hom_{+}(\g_{-},\g)$ such that:
	\[
	X+\chi(X)+\psi(X)\in \Hf\quad\text{for all }X\in\g_{-}.
	\]
	It is easy to see that $\psi$ is uniquely defined modulo $\Hom_{+}(\g_{-},\Hf^0)$. 
	
	Again, by considering the weights of the space $\Hom_{+}(\g_{-},\g)$, we see that none of them is proportional to $\alpha_1-2\alpha_2$. This implies that $\psi=0$, which completes the proof.
\end{proof}

It is not difficult to describe corresponding embedding $\varphi$ explicitly. It does already appear in \cite[Section 5.5]{DMM1} as embedding that corresponds to the following system of PDEs:
\begin{align*}
	Z_1^2u &= 0,\\
	Z_2^2u &= a Z_1u
\end{align*}
for any non-zero constant $a$. In fact, all such systems are equivalent via an appropriate rescaling. Note that $\Hf$ is isomorphic to $\hf = \gr \Hf$. Thus, we can assume that $Z_1$ and $Z_2$ can be written as follows in suitable local coordinates $(x,y,z)$:
\[
Z_1 = \frac{\partial}{\partial x} + \frac{1}{2}y\frac{\partial}{\partial z},\quad
Z_2 = \frac{\partial}{\partial y} - \frac{1}{2}x\frac{\partial}{\partial z}.
\]
 Then the above system of PDEs has an 8-dimensional solution space with the basis:
 \begin{align*}
 &1,\quad x+\tfrac{a}{2}y^2,\quad y,\\
 &xy + \tfrac{a}{6}y^3,\quad z+\tfrac{a}{12}y^3,\\
 &x\big(z-\frac{xy}{2}\big)+\tfrac{a}2y^2z+\tfrac{a}{12}xy^3+\tfrac{a^2}{60}y^5,\\
 &y\big(z+\frac{xy}{2}\big)+\tfrac{a}{12}y^4,\\
 &z^2-\frac{x^2y^2}{4}+\tfrac{a}{6}y^3z+\tfrac{a^2}{360}y^6.
 \end{align*}
This set of functions can be viewed as homogeneous coordinates of the embedding $(M,\f)\to P^3$, whose osculating flag corresponds to the embedding $\varphi\colon (M,\f)\to \Flag(V,\phi)$ from Proposition~\ref{p:cayley}.

This example can viewed as a contact generalization of the ruled Cayley's ruled cubic surface in $P^3$, which is (up to projective transformations) the only non-degenerate surface in $P^3$ with $3$-dimensional symmetry algebra.   

\section{Simply transitive embeddings with non-vanishing~$\chi_1^R$ and~$\chi_1^S$}
\label{sec:nonvanishing}
In this and the next section we complete the classification of all embeddings $\varphi\colon (M,\f)\to \Flag((V,\phi)$ with simply transitive group of automorphisms. We assume that $\dim \Hf=3$, so $\Hf^0=0$ and $H$ is (locally) simply transitive on the base.

\subsection{General setup}\label{ss:setup}
Similar to $\chi$, define another function $\psi\colon \g_{-}\to \g$ the commutative diagram:
\[
\begin{tikzcd}
	T_{\mathring{z}}Q\arrow[rr, "\omega_I"]\arrow[dr, "\omega_{-}"]& & \g \\
	& \g_{-} \cong \g / \phi^0\g \arrow[ur, "\psi"]  & 
\end{tikzcd}
\]

From the definition
\[
\psi(v) \equiv v \mod \phi^0\g\quad \text{ for }v\in\g_{-}.
\]
Using the identification $T_{\mathring{z}}\cong \Hf$ define another Lie bracket 
\[
\gamma\colon\g_{-}\times\g_{-}\to\g_{-}
\]
on $\g_{-}$ and decompose it as $\gamma = \sum_{l=0}^2\gamma_l$, where $\gamma_0$ coincides with the standard Lie bracket on $\g_{-}$:
\[
\gamma_0 = e_0\otimes e_2^*\wedge e_1^*.
\]

Finally, let $\varphi = \psi + \chi$ be a map $\g_{-}\to \lf$. Then the condition $d\omega+\tfrac12[\omega,\omega]=0$ is equivalent to
\[
\tfrac12[\varphi, \varphi] = \varphi.\gamma. 
\]

So far we have defined the maps 
\begin{align*}
	\psi &\in \Hom(\g_{-},\g),\\
	\chi &\in \Hom(\g_{-},\g^\perp),\\
	\gamma &\in \Hom(\wedge^2 \g_{-},\g_{-})
\end{align*}
using a fixed point $\mathring{z}\in Q$ and the identification of $T_{\mathring{z}}Q$ with $\Hf$. 

Taking $\mathring{z}$ to be an arbitrary point of $Q$, we can view them as maps
\begin{align*}
	\psi\colon &Q\to \Hom(\g_{-},\g),\\
	\chi \colon &Q\to \Hom(\g_{-},\g^\perp),\\
	\gamma\colon &Q\to \Hom(\wedge^2 \g_{-},\g_{-}).
\end{align*}
Due to the equivariance of $\omega$, it is easy to see that all these maps are also equivariant under the action of $\phi^0G$:
\begin{align*}
	\psi(z.a) &= \rho(a)^{-1}	\psi(z),\\
	\chi(z.a) &= \rho(a)^{-1}	\chi(z),\\
	\gamma(z.a) &= \rho(a)^{-1}	\gamma(z),
\end{align*}
for any $z\in Q, a\in \phi^0G$, where by $\rho(a)$ we understand the corresponding (different) representations of $\phi^0 G$.

From the above equivariance identities, it follows that for any $A_p\in\g_p$, $p>0$ we have
\[
R_{\exp(A_p)}^*\psi = \psi - \partial A_p + (\text{terms of degree }\ge p+1),
\]
where $\partial\colon \g_p=\Hom(\g_{-},\g)_p$ is the cochain differential. Since
\[
\Hom(\g_{-},\g) = \im\partial \oplus \ker \partial^*,
\]
for any $z\in Q$ we can choose uniquely $a\in \phi^1 G$ so that $\partial^* \psi_{za} = 0$.

Similar to $\gamma$, we decompose $\psi$ and $\chi$ by degree using the grading of~$\g$:
\begin{align*}
	\psi &= \sum_{p\ge 0} \psi_p,\quad \psi_0 = \operatorname{id};\\
	\chi &= \sum_{p\ge 1} \chi_p,\quad \chi_1 = \xi_1^R + \xi_1^S = R_{-1,1}\otimes e_2^* + S_{1,-1} \otimes e_1^*. 
\end{align*} 

From the fundamental equation we have:
\begin{align*}
	&\partial \psi_n + \tfrac12 \sum_{\substack{i+j=n \\ i,j>0}} [\psi_i,\psi_j] + \tfrac12 \sum_{\substack{i+j=n \\ i,j>0}} [\chi_i,\chi_j]_{I} = \sum_{\substack{i+j=n \\ i\ge 0,\, j>0}} \psi_i\gamma_j;\\
	&\partial \chi_n + \sum_{\substack{i+j=n \\ i,j>0}} [\psi_i,\chi_j] + \tfrac12 \sum_{\substack{i+j=n \\ i,j>0}} [\chi_i,\chi_j]_{II} = \sum_{\substack{i+j=n \\ i,j>0}} \chi_i\gamma_j;\\
	\intertext{and}
	&\partial^* \psi_n = 0.
\end{align*}

We are going to determine $\{ \psi_n, \chi_n, \gamma_n \}$ inductively for $n\ge 1$. For convenience sake we represent $\gamma$ as:
\begin{align*}
	\gamma_1 &= (e_0) \begin{pmatrix} P_1 & P_2 \end{pmatrix} \otimes \begin{pmatrix} e_0^* \wedge e_1^* \\ e_0^* \wedge e_2^* \end{pmatrix}
	+ \begin{pmatrix} e_1 & e_2 \end{pmatrix} \begin{pmatrix} p_1 \\ p_2 \end{pmatrix} \otimes (e_1^*\wedge e_2^*),\\
	\gamma_2 &= \begin{pmatrix} e_1 & e_2 \end{pmatrix} \begin{pmatrix} Q_{11} & Q_{12} \\ Q_{21} & Q_{22} \end{pmatrix} \otimes \begin{pmatrix} e_0^* \wedge e_1^* \\ e_0^* \wedge e_2^* \end{pmatrix}.
\end{align*}
Similarly, we write $\psi$ as:
\begin{align*}
	\psi_1 &= \begin{pmatrix} e_1 & e_2 \end{pmatrix} \begin{pmatrix} U_1 \\ U_2 \end{pmatrix} \otimes e_0^* + \begin{pmatrix} H_1 & H_2 \end{pmatrix} \begin{pmatrix} u_{11} & u_{12} \\ u_{21} & u_{22} \end{pmatrix} \otimes 
	\begin{pmatrix} e_1^* \\ e_2^* \end{pmatrix},\\
	\psi_2 &= \begin{pmatrix} H_1 & H_2 \end{pmatrix} \begin{pmatrix} V_1 \\ V_2 \end{pmatrix} \otimes e_0^* + \begin{pmatrix} \check{e}_1 & \check{e}_2 \end{pmatrix} \begin{pmatrix} v_{11} & v_{12} \\ v_{21} & v_{22} \end{pmatrix} \otimes \begin{pmatrix} e_1^* \\ e_2^* \end{pmatrix},\\
	\psi_3 &= \begin{pmatrix} \check{e}_1 & \check{e}_2 \end{pmatrix} \begin{pmatrix} W_1 \\ W_2 \end{pmatrix} \otimes e_0^* + \check{e}_0 \begin{pmatrix} w_{1} & w_{2} \end{pmatrix} \otimes 
	\begin{pmatrix} e_1^* \\ e_2^* \end{pmatrix},\\
	\psi_4 &= \zeta \check{e}_0\otimes e_0^*. 
\end{align*}
See Section~\ref{sec:letoile} for definition of the basis $\{e_0, e_1, e_2, H_1, H_2, \check{e}_0, \check{e}_1, \check{e}_2\}$ of~$\g$. 
 
\subsection{Degree 1}\label{nv-deg1} The fundamental equations in degree 1 are reduced to:
\begin{align*}
	&\partial \psi_1 = \gamma_1,\\
	&\partial^*\psi_1 = 0.
\end{align*} 
Direct computation shows
\begin{multline*}
	\partial \psi_1 = (e_0) \begin{pmatrix} U_2+u_{11}+u_{21} & -U_1+u_{12}+u_{22} \end{pmatrix} \otimes \begin{pmatrix} e_0^* \wedge e_1^* \\ e_0^* \wedge e_2^* \end{pmatrix} \\
	+ \begin{pmatrix} e_1 & e_2 \end{pmatrix} \begin{pmatrix} U_1+2u_{12}-u_{22} \\ U_2+u_{11}-2u_{21} \end{pmatrix} \otimes (e_1^*\wedge e_2^*)
\end{multline*}
and 
\[
\partial^* \psi_1 = \begin{pmatrix} \check{e_1} & \check{e_2}\end{pmatrix}\begin{pmatrix} U_2 \\ - U_1 \end{pmatrix} + \diag \begin{pmatrix} -2 & 1 \\ 1 & -2 \end{pmatrix}  
\begin{pmatrix} u_{11} & u_{12} \\ u_{21} & u_{22} \end{pmatrix}.
\]
So, we get:
\begin{align*}
	U_1 &= -\tfrac13(P_2-p_1),\\
	U_2 &= \tfrac13(P_1+p_2),\\
	\begin{pmatrix} u_{11} & u_{12} \\ u_{21} & u_{22} \end{pmatrix} &= \frac13 \begin{pmatrix} P_1 & P_2+p_1 \\ P_1-p_2 & P_2 \end{pmatrix}.
\end{align*}

Using the action of $\Aut_{+}(\g_{-})$ on $\psi$, we can always normalize $U_1=U_2=0$, which implies $p_1=P_2$ and $p_2=-P_1$.

\subsection{Degree 2}\label{nv-deg2} The fundamental equations in degree 2 are:
\begin{align}
	\label{deg2.1}	&\partial \chi_2 + [\psi_1,\chi_1]+\tfrac12 [\chi_1,\chi_1]_{II} = \chi_1\gamma_1,\\
	\label{deg2.2}	&\partial \psi_2 + \tfrac12[\psi_1, \psi_1]+\tfrac12[\chi_1,\chi_1]_I=\psi_1\gamma_1+\psi_0\gamma_2,\\
	\label{deg2.3}	&\partial^*\psi_2 = 0.
\end{align}

The terms of equation~\eqref{deg2.1} are
\begin{align*}
	\partial\chi_2 &= 4\begin{pmatrix} R_{-1,1} & S_{1,-1} \end{pmatrix} \begin{pmatrix} h_2^R \\ -h_2^S \end{pmatrix} e_1^*\wedge e_2^*,\\
	[\psi_1, \chi_1] &= \begin{pmatrix} U_1S_{0,-1} & U_2R_{-1,0} \end{pmatrix} \begin{pmatrix} e_0^* \wedge e_1^* \\ e_0^* \wedge e_2^* \end{pmatrix}
	-\begin{pmatrix} R_{-1,1} & S_{1,-1} \end{pmatrix} \begin{pmatrix} p_2 \\ p_1 \end{pmatrix} e_1^*\wedge e_2^*,\\
	[\chi_1,\chi_1]_{II} &= 0,\\
	\chi_1\gamma_1 &= \begin{pmatrix} R_{-1,1} & S_{1,-1} \end{pmatrix} \begin{pmatrix} p_2 \\ p_1 \end{pmatrix} e_1^*\wedge e_2^*.
\end{align*}
This implies that
\begin{align*}
	h_2^R &= \tfrac12 p_2 = -\tfrac12 P_1, \\
	h_2^S &= -\tfrac12 p_1 = -\tfrac12 P_2. 
\end{align*}

Next, the terms of equation~\eqref{deg2.2} are:
\begin{align*}
	\partial \psi_2 &= \begin{pmatrix} e_1 & e_2 \end{pmatrix} 
	\begin{pmatrix} -v_{21}-2V_1+V_2 & -v_{22} \\ v_{11} & v_{12}+V_1-2V_2 \end{pmatrix}\begin{pmatrix} e_0^*\wedge e_1^* \\ e_0^*\wedge e_2^*\end{pmatrix} \\ 
	& \qquad\qquad\qquad\qquad\qquad\qquad + \begin{pmatrix} H_1 & H_2 \end{pmatrix} \begin{pmatrix} V_1-v_{12} \\ V_2+v_{21} \end{pmatrix} e_1^*\wedge e_2^*, \\
	\tfrac12[\psi_1,\psi_1] &= 0 ,\\
	\tfrac12[\chi_1,\chi_1]_{I} &= 12\begin{pmatrix} H_1 & H_2 \end{pmatrix} \begin{pmatrix} -1 \\ 1\end{pmatrix} e_1^*\wedge e_2^*,\\
	\psi_1\gamma_1 &= \tfrac13p_1p_2 \begin{pmatrix} H_1 & H_2 \end{pmatrix} \begin{pmatrix} 1 \\ -1\end{pmatrix} e_1^*\wedge e_2^*,\\
	\psi_0\gamma_2 &= \begin{pmatrix} e_1 & e_2 \end{pmatrix} \begin{pmatrix} Q_{11} & Q_{12} \\ Q_{21} & Q_{22} \end{pmatrix} \otimes \begin{pmatrix} e_0^* \wedge e_1^* \\ e_0^* \wedge e_2^* \end{pmatrix}.
\end{align*}
Finally,
\[
\partial^* \psi_2 = (-(V_1+V_2)+v_{21}-v_{12})\check{e}_0.
\]

Solving these equations, we get
\begin{align}
	\label{vij} \begin{pmatrix} v_{11} & v_{12} \\ v_{21} & v_{22} \end{pmatrix} 
	&= 
	\begin{pmatrix} Q_{21} & -\frac14(3C+Q_{11}) \\
		-\frac14(3C+Q_{11}) & -Q_{12} \end{pmatrix},
	\\ \label{Vi}
	\begin{pmatrix} V_1 \\ V_2 \end{pmatrix} &= \tfrac14(C-Q_{11})\begin{pmatrix} 1 \\ -1 \end{pmatrix},
\end{align}
where 
\begin{equation}\label{eqC}
	C=12+\tfrac13p_1p_2=12-\tfrac13P_1P_2
\end{equation}
and in addition we have $Q_{11}+Q_{22}=0$.

This completes the computation of degree 2.

\subsection{Degree 3} The fundamental equations in degree 3 are:
\begin{align}
	\label{deg3.1}	&\partial \chi_3 + [\psi_2,\chi_1]+[\psi_1,\chi_2] = \chi_2\gamma_1+\chi_1\gamma_2,\\
	\label{deg3.2}	&\partial \psi_3 + [\psi_1, \psi_2]+[\chi_1,\chi_2]_I=\psi_2\gamma_1+\psi_1\gamma_2.
\end{align}

Note that starting from degree $n\ge 3$ the equations $\partial^*\psi_n = 0$ become trivial as $\g_n=0$ for $n\ge 3$.

Computing the individual terms of equation~\eqref{deg3.1} we get:
\begin{align*}
	\partial\chi_3 &= 
	\begin{pmatrix} R_{-1,1} & S_{1,-1} \end{pmatrix} 
	\begin{pmatrix} -h^R_3 & 0 \\ 0 & h_3^S \end{pmatrix}
	\begin{pmatrix} e_0^* \wedge e_1^* \\ e_0^* \wedge e_2^* \end{pmatrix} + 	
	3\begin{pmatrix} R_{0,1} & S_{1,0} \end{pmatrix} \begin{pmatrix} h_3^R \\ -h_3^S \end{pmatrix} e_1^*\wedge e_2^*,\\
	[\psi_2, \chi_1] &= \tfrac32(C-Q_{11}) \begin{pmatrix} R_{-1,1} & S_{1,-1} \end{pmatrix} \begin{pmatrix} 0 & -1 \\ 1 & 0 \end{pmatrix} \begin{pmatrix} e_0^* \wedge e_1^* \\ e_0^* \wedge e_2^* \end{pmatrix} 
	+ 3\begin{pmatrix} R_{0,1} & S_{1,0} \end{pmatrix} \begin{pmatrix} Q_{21} \\ Q_{12} \end{pmatrix} e_1^*\wedge e_2^*,\\
	[\psi_1, \chi_2] &= \tfrac12\begin{pmatrix} R_{-1,1} & S_{1,-1} \end{pmatrix} \begin{pmatrix} p_2^2 & p_1p_2 \\ -p_1p_2 & -p_1^2 \end{pmatrix} \begin{pmatrix} e_0^* \wedge e_1^* \\ e_0^* \wedge e_2^* \end{pmatrix} 
	+ \tfrac32\begin{pmatrix} R_{0,1} & S_{1,0} \end{pmatrix} \begin{pmatrix} -p_2^2 \\ p_1^2 \end{pmatrix} e_1^*\wedge e_2^*,\\
	\chi_2\gamma_1 &= \tfrac12\begin{pmatrix} R_{-1,1} & S_{1,-1} \end{pmatrix} \begin{pmatrix} -p_2^2 & p_1p_2 \\ -p_1p_2 & p_1^2 \end{pmatrix} \begin{pmatrix} e_0^* \wedge e_1^* \\ e_0^* \wedge e_2^* \end{pmatrix} 
	+ \tfrac32\begin{pmatrix} R_{0,1} & S_{1,0} \end{pmatrix} \begin{pmatrix} p_2^2 \\ -p_1^2 \end{pmatrix} e_1^*\wedge e_2^*,\\
	\chi_1\gamma_2 &= \tfrac12\begin{pmatrix} R_{-1,1} & S_{1,-1} \end{pmatrix} \begin{pmatrix} Q_{21} & Q_{22} \\ Q_{11} & Q_{12} \end{pmatrix} \begin{pmatrix} e_0^* \wedge e_1^* \\ e_0^* \wedge e_2^* \end{pmatrix}.
\end{align*}
Thus, we get
\begin{align*}
	\begin{pmatrix} -h_3^R+\tfrac12p_2^2 & -\tfrac32(C-Q_{11})+\tfrac12p_1p_2 \\ \tfrac32(C-Q_{11})-\tfrac12p_1p_2 & h_3^S - \tfrac12p_1^2 \end{pmatrix}
	&= \begin{pmatrix} -\tfrac12p_2^2 + Q_{21} & \tfrac12p_1p_2+Q_{22}\\ -\tfrac12p_1p_2+Q_{11} & \tfrac12p_1^2+Q_{12}\end{pmatrix},\\
	\begin{pmatrix} 3h_3^R +3Q_{21}-\tfrac32p_2^2 \\ -3h_3^S +3Q_{12}+\tfrac32p_1^2 \end{pmatrix}
	&= \begin{pmatrix} \tfrac32p_2^2 \\ -\tfrac32p_1^2 \end{pmatrix}.
\end{align*}
It is easy to see that the second of these equations is a consequence of the first one. Solving the first equation, we get:
\begin{align}
	\label{Qij} & Q_{11} = \tfrac35 C,\quad Q_{22} = -\tfrac35 C,\\
	& h_3^R = p_2^2 - Q_{21} = P_1^2 - Q_{21},\\
	& h_3^S = p_1^2 + Q_{12} = P_2^2 + Q_{12}. 
\end{align}

Similarly, computing the terms of equation~\eqref{deg3.2}, we get
\begin{align*}
	\partial \psi_3 &= \begin{pmatrix} H_1 & H_2 \end{pmatrix} 
	\begin{pmatrix} W_1-w_1 & -w_2 \\ -w_1 & W_2-w_2 \end{pmatrix}\begin{pmatrix} e_0^*\wedge e_1^* \\ e_0^*\wedge e_2^*\end{pmatrix} + \begin{pmatrix} \check{e}_1 & \check{e}_2 \end{pmatrix} \begin{pmatrix} W_1+w_1 \\ W_2+w_2 \end{pmatrix} e_1^*\wedge e_2^*, \\
	[\psi_1,\psi_2] &= \begin{pmatrix} \check{e}_1 & \check{e}_2 \end{pmatrix} \begin{pmatrix} -P_1Q_{21} \\ P_2Q_{12} \end{pmatrix},\\
	[\chi_1,\chi_2]_{I} &= 6\begin{pmatrix} H_1 & H_2 \end{pmatrix} 
	\begin{pmatrix} P_2 & -P_1 \\ -P_2 & P_1 \end{pmatrix}\begin{pmatrix} e_0^*\wedge e_1^* \\ e_0^*\wedge e_2^*\end{pmatrix} + 18\begin{pmatrix} \check{e}_1 & \check{e}_2 \end{pmatrix} \begin{pmatrix} P_2 \\ P_1 \end{pmatrix} e_1^*\wedge e_2^*, \\
	\psi_1\gamma_2 &= \tfrac13 \begin{pmatrix} H_1 & H_2 \end{pmatrix} 
	\begin{pmatrix} -P_2 & 2P_1 \\ -2P_2 & P_1 \end{pmatrix}
	\begin{pmatrix} Q_{11} & Q_{12} \\ Q_{21} & Q_{22} \end{pmatrix} 
	\begin{pmatrix} e_0^*\wedge e_1^* \\ e_0^*\wedge e_2^*\end{pmatrix},\\
	\psi_2\gamma_1 &= \begin{pmatrix} H_1 & H_2 \end{pmatrix} 
	\begin{pmatrix} P_1V_1 & P_2V_1 \\ P_1V_2 & P_2V_2 \end{pmatrix}
	\begin{pmatrix} e_0^*\wedge e_1^* \\ e_0^*\wedge e_2^*\end{pmatrix} + \begin{pmatrix} \check{e}_1 & \check{e}_2 \end{pmatrix} 
	\begin{pmatrix} v_{11} & v_{12} \\ v_{21} & v_{22} \end{pmatrix}
	\begin{pmatrix} P_1 \\ P_2 \end{pmatrix} e_1^*\wedge e_2^*,
\end{align*}
where $V_i$ and $v_{ij}$ were expressed via $C$ and $Q_{ij}$ in the degree 2 computation (see~\eqref{vij} and~\eqref{Vi}). Taking into account~\eqref{Qij}, we get
\begin{align*}
	& V_1 = \tfrac{C}{10}, \quad V_2 = -\tfrac{C}{10}, \\ 
	& v_{12} = v_{21} = -\tfrac{9C}{10}.
\end{align*}
Using this and the above expressions for the terms of equation~\eqref{deg3.2}, we get
\begin{equation}\label{eqsPQ}
	\begin{aligned}
		P_1Q_{11} + P_2Q_{21} &= 0,\\
		P_1Q_{12} + P_2Q_{22} &= 0
	\end{aligned}
\end{equation}
and
\begin{align*}
	w_1 &= P_1(6-\tfrac{C}{10}),\\
	w_2 &= P_2(-6+\tfrac{C}{10}),\\
	W_1 &= P_1(12-\tfrac{C}{5}),\\
	W_2 &= P_2(-12+\tfrac{C}{5}).
\end{align*}

\subsection{Degree 4} 
Proceeding as before, from the equation involving $\partial\chi_4$ we get:
\begin{align*}
	h_4^R &= P_1(-2P_1^2+5Q_{21}),\\
	h_4^S &= P_2(-2P_2^2-5Q_{12}).
\end{align*}

Similarly, from the equation on $\partial\psi_4$ we have
\begin{align}
	\label{eqQ21}	& Q_{21}\big(\tfrac{3}{5}C+12\big)+P_1^2\big(\tfrac{1}{5}C-24\big)=0,\\
	\label{eqQ12}	& Q_{12}\big(\tfrac{3}{5}C+12\big)-P_2^2\big(\tfrac{1}{5}C-24\big)=0,\\
	& \zeta = P_1P_2\big(33-\tfrac23C\big)+Q_{12}Q_{21}+\big(\tfrac9{10}C\big)^2.
\end{align}
Multiplying the first equation by $P_2$, the second equation by $P_1$, taking into account equations~\eqref{eqsPQ} and the fact that $P_1P_2=3(12-C)$ (see~\eqref{eqC}), we get
\begin{align*}
	P_1(C-15)(C-60)&=0,\\
	P_2(C-15)(C-60)&=0.
\end{align*}

Thus, we can distinguish three cases:
\begin{enumerate}
	\item[(i)] $P_1=P_2=0$ and $C=12$;
	\item[(ii)] $C=15$, $P_1P_2=-9$; 
	\item[(iii)] $C=60$, $P_1P_2=-144$.
\end{enumerate}

In all three cases equations \eqref{eqQ21}, \eqref{eqQ12} allow to determine $Q_{12}$ and $Q_{21}$ uniquely from $C$ and $P_1$, $P_2$:
\begin{enumerate}
	\item[(i)] $Q_{12}=Q_{21}=0$;
	\item[(ii)]  $Q_{21}=P_1^2$, $Q_{12}=-P_2^2$;
	\item[(iii)] $Q_{21}=\tfrac14 P_1^2$, $Q_{12}=-\tfrac14 P_2^2$.
\end{enumerate}
We note also that \eqref{Qij} determines uniquely $Q_{11}$ and $Q_{22}$ in terms of $C$.

\subsection{Degrees 5 and 6} 
From the equation on $\partial\chi_5$ we get:
\begin{align*}
	h_5^R &= -2\big(-2P_1^4+8P_1^2Q_{21}-3Q_{21}^2\big),\\
	h_5^S &= -2\big(2P_2^4+8P_2^2Q_{12}+3Q_{12}^2\big).
\end{align*}

Direct and lengthy computation shows that the equations on  $\partial\psi_5$ and $\partial\psi_6$ are automatically satisfied.

\emph{
	So, to sum up, the forms $\gamma$, $\psi$, $\chi$ that determine the osculating map $\varphi\colon (M,\f)\to \Flag(V,\phi)$ are expressed via four variables $P_1$, $P_2$, $Q_{12}$ and $Q_{21}$. The possible values of these variables split into three cases:
	\begin{enumerate}
		\item[($II_0$)] $P_1=P_2=0$, $Q_{12}=Q_{21}=0$ (corresponds to $C=12$);
		\item[($II_1$)] $P_1P_2=-9$, $Q_{21}=P_1^2$, $Q_{12}=-P_2^2$ (corresponds to $C=15$);
		\item[($II_2$)] $P_1P_2=-144$, $Q_{21}=\tfrac14 P_1^2$, $Q_{12}=-\tfrac14 P_2^2$  (corresponds to $C=60$).
	\end{enumerate}
}

Tensor $\gamma$ defines a new Lie algebra structure on $\g_{-}$ (viewed as a vector space), which is isomorphic to the symmetry algebra of $(P,\omega)$, or of the embedded manifold $\varphi(M)\subset \Flag(V,\phi)$. To distinguish this new Lie algebra structure from the graded Lie algebra structure on $\g_{-}$, we use the notation $\{Z_0, Z_1, Z_2\}$ for the basis of the symmetry algebra that corresponds to the basis $\{e_0,e_1,e_2\}$ of $\g_{-}$.

Explicitly, it has the following bracket relations in each of these cases:
\begin{enumerate}
	\item[($II_0$)] $[Z_0,Z_1]=\tfrac{36}{5}Z_1$, $[Z_0,Z_2]=-\tfrac{36}{5}Z_2$, $[Z_1,Z_2]=-Z_0$;
	\item[($II_1$)] $[Z_0,Z_1]=P_1(Z_0 - P_2 Z_1 + P_1 Z_2)$, $[Z_0,Z_2]=P_2(Z_0 - P_2 Z_1 + P_1 Z_2)$, $[Z_1,Z_2]= -Z_0 + P_2Z_1 - P_1Z_2$ ($P_1P_2=-9$);
	\item[($II_2$)] $[Z_0,Z_1]=P_1 (Z_0 - \tfrac{P_2}{4} Z_1 + \tfrac{P_1}{4} Z_2)$, $[Z_0,Z_2]=P_2 (Z_0 - \tfrac{P_2}{4} Z_1 + \tfrac{P_1}{4} Z_2)$, $[Z_1,Z_2]= -Z_0 + P_2Z_1 - P_1Z_2$ ($P_1P_2=-144$).
\end{enumerate}

The first of these Lie algebras is isomorphic to $\fsl(2,\R)$, while the other two are solvable with the 2-dimensional abelian derived algebra spanned by $e_0$ and $P_2 e_1 - P_1 e_2$. Note that Lie algebra structures for $(II_1)$ and $(II_2)$ are non-isomorphic, as in $(II_1)$ the Killing form vanishes identically, while in $(II_2)$ it has rank $1$. 

\subsection{Corresponding systems of PDEs}
Using the explicit description of $\om$ in terms of $\varphi=\psi+\chi$, we can also recover the corresponding systems of PDEs via the equation:
\begin{equation}\tag{\amgiS}
	d\eta + \om\cdot\eta = 0. 
\end{equation}

Let us see it more precisely for the case ($II_0$) and give a group theoretic interpretation of the associated differential equation. We have:
\begin{align*}
	\varphi(Z_0) &= e_0 + \tfrac{6}{5} H_1 - \tfrac{6}{5} H_2 + \left(\tfrac{54}{5}\right)^2 \check e_0,\\
	\varphi(Z_1) &= e_1 - \tfrac{54}{5} \check e_2 + S_{1,-1}, \\
	\varphi(Z_2) &= e_2 - \tfrac{54}{5} \check e_1 + R_{-1,1}.
\end{align*}

These formulas define an embedding of the Lie algebra $\fsl(2)$ to $\fso(5,3)\subset \fsl(8)$ and, thus, define the structure of $\fsl(2)$-module on $\R^8$. To identify the decomposition of this module into the irreducible submodules, it is sufficient to inspect the eigenvalues of the image of $Z_0$, which spans the Cartan subalgebra in $\fsl(2)$. Direct computation shows that $\varphi(Z_0)$ has the following eigenvalues:
\[
\left(\tfrac{108}{5}, \tfrac{72}{5}, \tfrac{36}{5}, 0, 0, -\tfrac{36}{5}, -\tfrac{72}{5}, -\tfrac{108}{5}\right), 
\]
which up to the constant are equal to $(6,4,2,0,0,-2,-4,-6)$. This implies that $\R^8$ is decomposed into the sum of two irreducible $\fsl(2)$-submodules of dimensions $7$ and $1$. 

The osculating embedding in this case can be viewed purely in terms of the representation theory of $\sll(2)$. Namely, the osculating embedding under consideration is an osculating map for some equivariant embedding $SL(2,\R)\to P^7$, where the action of $SL(2,\R)$ on $P^7$ is determined from the decomposition of $\R^8=V_7+V_1$ into two irreducible modules of dimensions $7$ and $1$. Due to equivariancy the embedding $SL(2,\R)\to P^7$ is uniquely determined by its value at the identity, which is a certain 1-dimensional subspace $V^{-1}\subset V$. 

Let $\{e_0,\dots,e_6\}$ be the standard basis of the corresponding $\sll(2,\R)$-module $V_7$ and let $f$ be the basis element in $V_1$ equipped with the trivial action of $\sll(2)$. Define $V^{-1}$ as:
\[
V^{-1} = \langle e_0 + \sqrt{10}\,e_3 + e_6 + f\rangle.
\]
Define also the contact filtration on $\fsl(2)$ by $\fsl(2)^{-1}=\langle e_1, e_2\rangle$, which corresponds to the subspace $\left\{\left(\begin{smallmatrix} 0 & x \\ y & 0 \end{smallmatrix}\right)\right\}$. 

This extends to the filtration of $V$ by $V^{-i-1}=V^{-i}+\fsl(2)^{-1}V^{-i}$. It is easy to check that
\[ 
\dim V^{-2}=3,\ \dim V^{-3}=5,\ \dim V^{-4}=7,\ \dim V^{-5}=8,
\] 
so $V^{-5}=V$. These dimensions correspond the symbol $\fsl_3$ of the embedding. One can check that these conditions uniquely determine the subspace $V^{-1}$ up to equivalence.

Finally, let us write the corresponding system of PDEs explicitly in some natural coordinate system on $SL(2,\R)$. Denote also by $Z^*_i$ ($i=0,\dots,2$), the left-invariant vector field on the symmetry group $H$ corresponding to the basis element $Z_i$. Then by definition we have $\om(Z^*_i)=\phi(Z_i)$. For simplicity, we shall use the same notation $Z_i$ for both basis elements in the symmetry algebra $\mathfrak{H}$ and the corresponding left-invariant vector fields $Z_i^*$ on $H$. 

Equation~(\amgiS) takes the form:
\[
Z_i\eta + \phi(Z_i)\eta = 0,\quad i=0,1,2.
\]
Using the above formulas for $\phi(e_i)$, we derive the following system of PDEs:
\begin{align*}
	Z_1^2u &= -6Z_2u,\\
	Z_2^2u &= 6Z_1u,
\end{align*}
where $u$ is the last coordinate of $\eta$ in the basis $\{A_1,\dots,A_8\}$.

The group $H$ is locally isomorphic to $SL(2,\R)$, which acts locally simply transitively on the projectivized cotangent bundle to $P^1\times P^{1,*}$.  We can choose local coordinate system $(x,y,z)$ on $H$ such that up to non-zero scales vector fields $Z_1, Z_2, Z_0$ have the form:
\begin{align*}
	Z_1 &= \partial_x + y^2\partial_y + y\partial_z,\\
	Z_2 &= x^2\partial_x + \partial_y - x\partial_z,\\
	Z_0 &= - [Z_1, Z_2] = -2(x\partial_x - y \partial_y + \partial_z).
\end{align*}
Here $x$ and $y$ are affine coordinates on the two copies of $P^1$, and $z=\log (dy/dx)$. Then, due to the scaling factors, the above system of PDEs is transformed to the following one:
\begin{equation}\label{eq-case1}
	\begin{aligned}
		Z_1^2u + \sqrt{10}\, Z_2 u &= 0,\\
		Z_2^2u - \sqrt{10}\, Z_1 u &= 0.
	\end{aligned}
\end{equation}
It has an 8-dimensional solution space which can be described as follows. First, note that the Lie algebra of right-invariant vector fields on $SL(2,\R)$ is spanned in the chosen coordinate system by the following vector fields:
\begin{align*}
	Z_1' &= e^{z} \Big((xy+1)\partial_y + x\partial_z\Big),\\ 
	Z_2' &= e^{-z} \Big(-(xy+1)\partial_x + y\partial_z\Big),\\
	Z_0'&=  \partial_z.
\end{align*}
It is clear that it lies in the symmetry algebra of~\eqref{eq-case1} and thus preserves its solution space. Explicit computation shows that the solution space of~\eqref{eq-case1} is spanned by constants and the following 7-dimensional vector space invariant with respect to the action of $Z_0', Z_1', Z_2'$:
\[
\quad (Z_2')^k \left[\frac{x^6+\sqrt{10}x^3+1}{(xy+1)^3} e^{3z}\right], \quad k=0,\dots,6. 
\]
This reconfirms the decomposition of the action of $\mathfrak{H}$ on $\R^8$ into the sum of 1-dimensional and 7-dimensional irreducible subspaces. 

Cases ($II_1$) and ($II_2$) can be treated in a similar manner. This results in the following systems of PDEs.

Case ($II_1$):
\begin{align*}
	(Z_1-P_1)^2u &= -6(Z_2-P_2)u + (P_1^2+3P_2)u,\\
	(Z_2-P_2)^2u &= 6(Z_1-P_1)u+(P_2^2-3P_1)u,
\end{align*}
where $P_1P_2=-9$.

Case ($II_2$):
\begin{align*}
	(Z_1-P_1)^2u &= -6(Z_2-P_2)u + (\tfrac{1}{4}P_1^2+3P_2)u,\\
	(Z_2-P_2)^2u &= 6(Z_1-P_1)u+(\tfrac{1}{4}P_2^2-3P_1)u,
\end{align*}
where $P_1P_2=-144$.

\section{Simply transitive embeddings with only one non-vanishing $\chi_1^R$ or $\chi_1^S$}\label{s:vanishing}
Without loss of generality we can assume that $\chi_1^S=0$ and $\chi_1^R\ne0$. From Proposition~\ref{hRS} we have $h_k^S=0$ for all $k=1,\dots,5$.

As in Subsection~\ref{ss:setup} we define
\begin{align*}
	\psi\colon &Q\to \Hom(\g_{-},\g);\\
	\chi \colon &Q\to \Hom(\g_{-},\g^\perp);\\
	\gamma\colon &Q\to \Hom(\wedge^2 \g_{-},\g_{-}),
\end{align*}
decompose them as:
\begin{align*}
	\psi &= \sum_{p\ge 0} \psi_p,\quad \psi_0 = \operatorname{id};\\
	\chi &= \sum_{p\ge 1} \chi_p,\quad \chi_1 = \xi_1^R + \xi_1^S = R_{-1,1}\otimes e_2^* + S_{1,-1} \otimes e_1^*;\\
	\gamma &= \sum_{p\ge 0} \gamma_p,\quad \gamma_0 =  e_0\otimes e_2^*\wedge e_1^*,
\end{align*}
and proceed determining $\{ \psi_n, \chi_n, \gamma_n \}$ inductively for $n\ge 1$.

\subsection{Degree 1} Here the computation is identical to the non-vanishing case (see Subsection~\ref{nv-deg1}), except that we have $h_1^R=1$, $h_1^S=0$. The coefficients $U_i$ and $u_{ij}$ are completely determined by $p_i$ and $P_j$ by:
\begin{align*}
	U_1 &= -\tfrac13(P_2-p_1),\\
	U_2 &= \tfrac13(P_1+p_2),\\
	\begin{pmatrix} u_{11} & u_{12} \\ u_{21} & u_{22} \end{pmatrix} &= \frac13 \begin{pmatrix} P_1 & P_2+p_1 \\ P_1-p_2 & P_2 \end{pmatrix}.
\end{align*}

As in Subsection~\ref{nv-deg1}, we can assume $U_1=U_2=0$ and thus $p_1=P_2$, $p_2=-P_1$.

\subsection{Degree 2} 
As in the non-vanishing case (see Subsection~\ref{nv-deg2}), we have:
\[
h_2^R = -\tfrac{1}{2}P_1,
\]
and the coefficient $h_2^S$ vanishes due to Proposition~\ref{hRS}.

Next, we compute coefficients $v_{ij}$ and $V_i$ in terms of $P_j$, $Q_{ij}$:
\begin{align*}
v_{11} &= Q_{21},\\
v_{22} &= -Q_{12},\\
v_{12} &= v_{21}=\tfrac{1}{4}(P_1P_2-Q_{11}),\\
V_{2} &= -V_{1} = \tfrac{1}{12}(P_1P_2+3Q_{11}).
\end{align*}
In addition, as in Subsection~\ref{nv-deg2} we also have $Q_{11}+Q_{22}=0$.

\subsection{Degree 3 and higher}\label{ss:vanishing:deg3} Further computation shows that:
\begin{align*}
	h_3^R &= P_1^2-Q_{21},\\
	Q_{11} &= - Q_{22} = -\tfrac{1}{5}P_1P_2.
\end{align*}
The parameters $w_i$, $W_j$ are explicitly computed as:
\begin{align*}
w_1 &= \tfrac{1}{10}P_1^2P_2-\tfrac{1}{3}P_2Q_{21},\\
w_2 &= -\tfrac{1}{10}P_1P_2^2-\tfrac{1}{3}P_1Q_{12}
\end{align*}
 and
\begin{align*}
W_1 &= \tfrac{1}{3}P_2Q_{21},\\
W_2 &= \tfrac{1}{3}P_1Q_{12}.
\end{align*}

Furthermore, we get the following equations on $Q_{12}$ and $Q_{21}$:
\begin{align*}
P_2(P_1^2-5Q_{21})&=0,\\
P_1(P_2^2+5Q_{12})&=0.
\end{align*}
So, assuming that both $P_1$ and $P_2$ do not vanish, we can express $Q_{12}$ and $Q_{21}$ via $P_1$, $P_2$. However, proceeding with degree 4 computations, we get $P_1^3P_2=P_1P_2^3=0$, which contradicts to the assumption that both $P_1\ne 0$ and $P_2\ne 0$. 

So, we now have to consider three subcases.

\subsubsection*{Subcase 1. $P_1=0$, $P_2\ne0$} This implies immediately that $Q_{21}=0$, which completes the analysis of all parameters of degree 3. 

In degree 4 we get $h_4^R=\zeta=0$. From equations in degree 5 we also get $h_5^R=0$. And we have no further restrictions in degree 6. Thus, we get a family of embeddings parametrizied by $P_2$ and $Q_{12}$.

\subsubsection*{Subcase 2. $P_1\ne 0$, $P_2=0$} This implies immediately that $Q_{12}=0$, which completes the analysis of all parameters of degree 3. In degree 4 we get:
\begin{align*}
	\zeta &= 0,\\
	h_4^R &= P_1(5Q_{21}-2P_1^2).
\end{align*}
Proceeding to degree 5, we get:
\[
	h_5^R = 4P_1^4-16P_1^2Q_{21}+6Q_{21}^2.
\]
Finally, computations in degree 6 lead to the following equation:
\[
	(P_1^2-Q_{21})(P_1^2-4Q_{21})=0.
\]
So, we get $Q_{21}=P_1^2$ or $Q_{21}=\frac{1}{4}P_1^2$.  Thus, we get a family of embeddings parametrizied by $P_1$. This completes the analysis of this subcase. 

\subsubsection*{Subcase 3. $P_1=P_2=0$} We immediately get $h_4^R=0$ and 
\[
	\zeta=Q_{12}Q_{21}. 
\]
Proceeding to degree 5, we find $h_5^R=6Q_{21}^2$. Finally, degree 6 equations imply that $Q_{21}=0$ and thus we get $\zeta=h_5^R=0$. Thus, we arrive at the same equations as in Subcase 1 ($P_1=0$), but with extra relation $P_2=0$.

\subsection{Special values of parameters}
If parameters $P_1,P_2,Q_{12}$ vanish identically, this leads us exactly to the model with transitive symmetry algebra as defined in Proposition~\ref{p:cayley}. However, this is not the only case, when this may happen. Namely, let us determine under which conditions of parameters the resulting 3-dimensional subalgebra $\Hf$ may be included into a bigger subalgebra $\widehat \Hf$ such that $\Hf+\g=\widehat \Hf+\g$ and they define the same structure function $\chi\in\Hom(\g_{-},\g^\perp)$.

From Subsection~\ref{ss_stab} we know that this is possible only if $\dim \widehat\Hf=4$ and up to $\exp(\g_{+})$ the subalgebra $\widehat\Hf$ is conjugate to the symmetry algebra of the contact Cayley surface given in Proposition~\ref{p:cayley}:
\[
\langle e_0, e_1, e_2+R_{-1,1}, 5 H_1 + 4 H_2 \rangle
\] 

In particular, choosing an arbitrary 3-dimensional subalgebra in this symmetry algebra complementary to the stabilizer $\langle 5H_1+4H_2\rangle $, up to the action of $\exp(\g_{+})$ we get one of the subcases classified in Subsection~\ref{ss:vanishing:deg3}. 

Elementary calculations show that there are two 1-parameter families of such subalgebras:
\[
\langle e_0, e_1+\alpha(5 H_1 + 4 H_2), e_2+R_{-1,1}\rangle,\quad \alpha\in \R
\]
and
\[
\langle e_0, e_1, e_2+R_{-1,1}+\beta(5 H_1 + 4 H_2)\rangle,\quad \beta\in \R.
\]
To identify them among the cases of Section~\ref{s:vanishing}, it is sufficient to check if any of the 3-dimensional subalgebras $\Hf\subset \fso(5,3)$ determined in subcases 1,2 and 3 of Subsection~\ref{ss:vanishing:deg3} can be complemented to a 4-dimensional subalgebra by an element $H$ of the form $\exp(\g_{+})(5H_1+4H_2)$, or, explicitly:
\[
H = 5H_1 + 4H_2 + a_1 \check e_1 + a_2 \check e_2 + a_0 \check e_0
\]
for some constants $a_0,a_1,a_2\in\R$.

Simple calculation shows that:
\begin{itemize}
	\item in subcases 1 and 3 (treated together) such element $H$ exists if and only if $Q_{12}+\tfrac{P_2^2}{25}=0$;
	\item in subcase 2 such element $H$ exists if and only if $Q_{21}=\tfrac{1}{4}P_1^2$. 
\end{itemize}
The explicit form of such element $H$ is given in Section~\ref{s:summary}.
So, we can exclude these cases from the final list of results.

We note that the normalization condition $h_1^R=1$ still leaves a freedom in scaling of the remaining parameters that results in the action of the following one-parameter subgroup in $SO(5,3)$:
\[
\diag(t^3,t^2,t,1,1,1/t,1/t^2,1/t^3), \quad t\in\R^* 
\]
in the basis $A_1,\dots,A_8$. It acts on the parameters $P_1$, $P_2$, and $Q_{12}$ as follows:
\[
(P_1,P_2,Q_{12})\mapsto (t^2P_1, tP_2, t^2Q_{12})).
\]  
So, in combined cases 1 and 3 we can assume that $Q_{12}+\tfrac{P_2^2}{25}\ne 0$ and normalize this expression to $\pm 1$. Similarly, in case 2 we assume that $P_1\ne 0$ and also normalize it to $\pm1$.

The corresponding symmetry algebras and systems of PDEs are given in Section~\ref{s:summary}.  

\section{Summary of the results}
\label{s:summary}
We use the following notation in this section:
\begin{itemize}
	\item The basis of the symmetry algebra is denoted by $\{Z_0,Z_1,Z_2\}$ in simply transitive cases (all cases except $(O)$ and $(I_0)$). In case $(I_0)$ the basis is $\{H,Z_0,Z_1,Z_2\}$, where $\{H\}$ is the basis of the stationary subalgebra.
	\item The embedding of the symmetry algebra into $\fso(5,3)$ uses the basis $\{e_0,e_1,e_2,H_1,H_2,\check e_0, \check e_1, \check e_2\}$ of the subalgerba $\fsl(3,\R)\subset \fso(5,3)$ and the basis elements $R_{i,j}$, $S_{i,j}$ for two irreducible $\fsl(3,\R)$ submodules, as defined in Section~\ref{sec:letoile}.
	\item The systems of PDEs use the same notation $Z_1$ $Z_2$ for the left-invariant vector fields on the symmetry group, and $u$ denotes an unknown function on this group. 
\end{itemize}

\smallskip\noindent
$\mathbf{(O)}$ Symmetry algebra: $\fsl(3,\R)$.

Embedding: $\ad\colon \fsl(3,\R)\to \fso(5,3)\subset \fsl(8,\R)$.

Equation: $Z_1^2u=Z_2^2u=0$, where   
\[
Z_1 =\tfrac{\partial}{\partial x}-\tfrac{y}{2} \tfrac{\partial}{\partial z},
\quad Z_2 =\tfrac{\partial}{\partial y}+\tfrac{x}{2} \frac{\partial}{\partial z}.
\]

\smallskip\noindent
$\mathbf{(I_0)}$ Symmetry algebra:
\[
[Z_1,Z_2] = -Z_0,\quad [H,Z_0]=-3Z_1,\quad [H,Z_1]=-2Z_1,\quad [H,Z_2]=-Z_2.
\]
Embedding:
\begin{align*}
	Z_0 &\mapsto e_0,\\
	Z_1 &\mapsto e_1,\\
	Z_2 &\mapsto e_2 + R_{-1,1},\\
	H &\mapsto \tfrac{5}{3}H_1 + \tfrac{4}{3}H_2.
\end{align*}
Equation:
\begin{align*}
	Z_1^2u&=0,\\
	Z_2^2u&=6Z_1u.
\end{align*}
This is the special case of $(I_1)$ for $P_2=Q_{12}=0$.

\smallskip\noindent
$\mathbf{(I_1)}$ ($P_1=0$) Symmetry algebra:
\[
[Z_1,Z_2] = -Z_0 + P_2 Z_1,\quad [Z_0,Z_2] = P_2 Z_0 + Q_{12} Z_1.
\]
where $P_2,Q_{12}$ do not vanish simultaneously and are viewed up to the scaling
\[
(P_2,Q_{12})\mapsto (tP_2,t^2Q_{12}),\quad t\in \R^*.
\]

Embedding:
\begin{align*}
	Z_0 &\mapsto e_0,\\
	Z_1 &\mapsto e_1,\\
	Z_2 &\mapsto e_2 +\tfrac{2P_2}{3}H_1 + \tfrac{P_2}{3}H_2 - Q_{12}\check e_2 + R_{-1,1}.
\end{align*}

Equation:
\begin{align*}
Z_1^2u&=0,\\
Z_2^2u&=6Z_1u+2P_2Z_2u-(Q_{12}+P_2^2)u.
\end{align*}

This case admits an additional 4-th symmetry if and only if $Q_{12}+\tfrac{P_2^2}{25}=0$. If this relation is satisfied then this additional symmetry is represented by the matrix:
\[
H = \tfrac{5}{3}H_1 + \tfrac{4}{3}H_2 + \tfrac{P_2}{5}\check e_2. 
\]
So, we can assume that $Q_{12}+\tfrac{P_2^2}{25}\ne 0$ and normalize this expression to $\pm 1$, or $Q_{12}=\pm 1 - \tfrac{P_2^2}{25}$.

\smallskip\noindent
$\mathbf{(I_2)}$ ($Q_{21}=P_1^2, P_1\ne 0$) Symmetry algebra:
\begin{align*}
[Z_1,Z_2] &= -Z_0 - P_1 Z_2,\\
[Z_0,Z_1] &= P_1 (Z_0+ P_1 Z_2), \\
[Z_0,Z_2] &= 0
\end{align*}
The parameter $P_1\ne 0$ is viewed up to the scaling $P_1\to t^2P_1$, $t\in \R^*$ and can be normalized to $P_1=\pm 1$.

Embedding:
\begin{align*}
	Z_0 &\mapsto e_0-\tfrac{P_1}{2}R_{-1,1}+3P_1^3R_{1,1}-6P_1^4R_{2,1},\\
	Z_1 &\mapsto e_1+\tfrac{P_1}{3}(H_1+2H_2)+P_1^2\check e_1,\\
	Z_2 &\mapsto e_2+R_{-1,1}-\tfrac{3P_1}{2}R_{0,1}+3P_1^3 R_{2,1}.
\end{align*}

Equation:
\begin{align*}
&Z_1(Z_1-2P_1)u =0,\\
&Z_2^2u =6Z_1u-9P_1u.
\end{align*}

\smallskip\noindent
$\mathbf{(I_2')}$ ($Q_{21}=\tfrac{1}{4}P_1^2, P_1\ne 0$) Symmetry algebra:
\begin{align*}
[Z_1,Z_2] &= -Z_0 - P_1 Z_2, \\
[Z_0,Z_1] &= P_1 (Z_0 + \tfrac{P_1}{4} Z_2),\\
[Z_0,Z_2] &= 0.\\
\end{align*}

Embedding:
\begin{align*}
	Z_0 &\mapsto e_0-\tfrac{P_1}{2}R_{-1,1}+\tfrac{3P_1^2}{4}R_{0,1}-\tfrac{3P_1^3}{4}R_{1,1}+\tfrac{3P_1^4}{8}R_{2,1},\\
	Z_1 &\mapsto e_1+\tfrac{P_1}{3}(H_1+2H_2)+\tfrac{P_1^2}{4}\check e_1,\\
	Z_2 &\mapsto e_2+R_{-1,1}-\tfrac{3P_1}{2}R_{0,1}++\tfrac{3P_1^2}{2}R_{1,1}-\tfrac{3P_1^3}{4} R_{2,1}.
\end{align*}

Equation:
\begin{align*}
&(2Z_1-P_1)(2Z_1-3P_1)u =0,\\
&Z_2^2u =6Z_1u-9P_1u.
\end{align*}

This case always admits an additional symmetry represented by the matrix
\[
 H = \tfrac{5}{3}H_1 + \tfrac{4}{3}H_2 + P_1 \check e_1.
\]
and is excluded from the final list of results.

\smallskip\noindent
$\mathbf{(II_0)}$ Symmetry algebra: 
\[
[Z_0,Z_1]=\tfrac{36}{5}Z_1,\quad [Z_0,Z_2]=-\tfrac{36}{5}Z_2,\quad [Z_1,Z_2]=-Z_0.
\]

Embedding:
\begin{align*}
	Z_0&\mapsto e_0 + \tfrac{6}{5}(H_1-H_2)+\tfrac{2916}{25}\check e_0,\\
	Z_1&\mapsto e_1+S_{1,-1}-\tfrac{54}{5}\check e_2,\\
	Z_2&\mapsto e_2+R_{-1,1}-\tfrac{54}{5}\check e_2.
\end{align*}

Equation: 
\begin{align*}
	Z_1^2u &=-6Z_2u,\\
	Z_2^2u &=6Z_1u.
\end{align*}

\smallskip\noindent
$\mathbf{(II_1)}$ Symmetry algebra:
\begin{align*}
	[Z_0,Z_1]&=P_1(Z_0 - P_2 Z_1 + P_1 Z_2),\\
	[Z_0,Z_2] &=P_2(Z_0 - P_2 Z_1 + P_1 Z_2),\\
	[Z_1,Z_2]&= -Z_0 + P_2Z_1 - P_1Z_2,
\end{align*}
where $P_1P_2=-9$.

Embedding:
\begin{align*}
Z_0&\mapsto e_0 - \tfrac{1}{2}(P_1R_{-1,1}-P_2S_{1,-1})+3(P_1^3R_{1,1}-P_2^3S_{1,1})-6(P_1^4R_{2,1}-P_2^4S_{1,2}) \\
&\qquad\qquad + \tfrac{3}{2}(H_1-H_2)+9(P_1\check e_1-P_2\check e_2)-\tfrac{567}{4}\check e_0,
\\
Z_1&\mapsto e_1+S_{1,-1}-\tfrac{3P_2}{2}S_{1,0}-3P_2^3S_{1,2}+\tfrac{P_1}{3}(H_1+2H_2)+P_1^2\check e_1-\tfrac{27}{2}\check e_2 +\tfrac{9P_1}{2}\check e_0,\\
Z_2&\mapsto e_1+R_{-1,1}-\tfrac{3P_1}{2}R_{0,1}+3P_1^2R_{2,1}+\tfrac{P_2}{3}(2H_1+H_2)-\tfrac{27}{2}\check e_1+P_2^2\check e_2-\tfrac{9P_2}{2}\check e_0.
\end{align*}

Equation:
\begin{align*}
	(Z_1-P_1)^2u &=-6(Z_2-P_2)u+(P_1^2+3P_2)u,\\ 
	(Z_2-P_2)^2u &=6(Z_1-P_1)u+(P_2^2-3P_1)u.
\end{align*}

\smallskip\noindent
$\mathbf{(II_2)}$ Symmetry algebra:
\begin{align*}
	[Z_0,Z_1] &=P_1 (Z_0 - \tfrac{P_2}{4} Z_1 + \tfrac{P_1}{4} Z_2),\\
	[Z_0,Z_2] &=P_2 (Z_0 - \tfrac{P_2}{4} Z_1 + \tfrac{P_1}{4} Z_2),\\
	[Z_1,Z_2] &= -Z_0 + P_2Z_1 - P_1Z_2,
\end{align*}
where $P_1P_2=-144$.

Embedding:
\begin{align*}
	Z_0&\mapsto e_0 - \tfrac{1}{2}(P_1R_{-1,1}-P_2S_{1,-1})+\tfrac{3}{4}(P_1^2R_{0,1}-P_2^2S_{1,0})
	+\tfrac{3}{4}(P_1^3R_{1,1}-P_2^3S_{1,1}) \\
	&\qquad\qquad
	+\tfrac{3}{8}(P_1^4R_{2,1}-P_2^4S_{1,2})  + 6(H_1-H_2)+324\check e_0,
	\\
	Z_1&\mapsto e_1+S_{1,-1}-\tfrac{3P_2}{2}S_{1,0}+\tfrac{3P_2^2}{2}S_{1,1}-\tfrac{3P_2^3}{4}S_{1,2}+\tfrac{P_1}{3}(H_1+2H_2)+\tfrac{P_1^2}{4}\check e_1-54\check e_2,\\
	Z_2&\mapsto e_1+R_{-1,1}-\tfrac{3P_1}{2}R_{0,1}+\tfrac{3P_1^2}{2}R_{1,1}-\tfrac{3P_1^3}{4}R_{2,1}+\tfrac{P_2}{3}(2H_1+H_2)-54\check e_1+\tfrac{P_2^2}{4}\check e_2.
\end{align*}

Equation:
\begin{align*}
	(Z_1-P_1)^2u &=-6(Z_2-P_2)u+(\tfrac14 P_1^2+3P_2)u,\\ 
	(Z_2-P_2)^2u &=6(Z_1-P_1)u+(\tfrac14 P_2^2-3P_1)u.
\end{align*}

\vskip8mm
\begin{minipage}{\textwidth}\noindent
	Boris DOUBROV, \\
	Belarusian State University,\\
	Nezavisimosti ave. 4, 220030 Minsk, Belarus\\
	Email:  doubrov@bsu.by
\end{minipage}

\vskip8mm
\begin{minipage}{\textwidth}\noindent
	Tohru MORIMOTO,\\
	Institut Kiyoshi Oka de Math\'ematiques,\\
	Nara Women's University, \\
	Nara 630-8506, Japan\\
	Email: morimoto@cc.nara-wu.ac.jp
\end{minipage}

\end{document}